\documentclass[a4paper]{article}
\usepackage[all]{xy}\usepackage[latin1]{inputenc}        
\usepackage[dvips]{graphics,graphicx}
\usepackage{amsfonts,amssymb,amsmath,color,mathrsfs, amstext}
\usepackage{amsbsy, amsopn, amscd, amsxtra, amsthm,authblk}
\usepackage{enumerate,algorithmicx,algorithm}
\usepackage{algpseudocode}
\usepackage{upref}
\usepackage{bm}
\usepackage{geometry}
\usepackage{amsthm,amsmath,amssymb}
\usepackage{relsize}
\usepackage{mathrsfs}
\usepackage{exscale}
\usepackage{graphicx}
\usepackage{tabularx}
\usepackage{threeparttable,multirow,dcolumn,booktabs}
\usepackage{algorithmicx,algorithm}
\usepackage{tabularx}
\usepackage{caption}
\usepackage[displaymath]{lineno}
\usepackage{float,color}
\usepackage{bm}
\usepackage{caption}
\usepackage{setspace}
\usepackage{yhmath}
\usepackage{booktabs}
\usepackage{subcaption}
\usepackage{multirow}
\usepackage{makecell}
\usepackage[colorlinks,
linkcolor=red,
anchorcolor=red,
citecolor=red
]{hyperref}

\numberwithin{equation}{section}

\newcommand{\V}[1]{\boldsymbol{#1}} 
\newcommand{\abs}[1]{\left|#1\right|} 

\renewcommand{\i}{\mathsf i} 

\newcommand{\te}{\text}
\makeatletter
\newcommand{\Biggg}{\bBigg@{3.5}}
\newtheorem{theorem}{Theorem}[section] 
\newtheorem{lemma}[theorem]{Lemma} 
\newtheorem{remark}[theorem]{Remark}

\usepackage{hyperref}
\usepackage{algorithm}
\usepackage{setspace}
\usepackage{color}

\begin{document}

\title{Random Batch Ewald Method for Dielectrically Confined Coulomb Systems}

\author[1,2]{Zecheng Gan\thanks{zechenggan@ust.hk}}
\author[1,3]{Xuanzhao Gao\thanks{xz.gao@connect.ust.hk}}
\author[4]{Jiuyang Liang\thanks{liangjiuyang@sjtu.edu.cn}}
\author[4]{Zhenli Xu\thanks{xuzl@sjtu.edu.cn}}

\affil[1]{Thrust of Advanced Materials, The Hong Kong University of Science and Technology (Guangzhou), Guangdong, China}
\affil[2]{Department of Mathematics, The Hong Kong University of Science and Technology, Clear Water Bay, Kowloon, Hong Kong SAR}
\affil[3]{Department of Physics, The Hong Kong University of Science and Technology, Clear Water Bay, Kowloon, Hong Kong SAR}
\affil[4]{School of Mathematical Sciences, MOE-LSC and CMA-Shanghai, Shanghai Jiao Tong University, Shanghai, China}

\date{}
\maketitle

\begin{abstract}
Quasi two-dimensional Coulomb systems
have drawn widespread interest.
The reduced symmetry of these systems leads to complex collective behaviors,
yet simultaneously poses significant challenges for particle-based simulations.
In this paper, a novel method is presented for efficiently simulate a collection of charges confined in doubly-periodic slabs, with the extension to scenarios involving dielectric jumps at slab boundaries.
Unlike existing methods, the method is insensitive to the aspect ratio of simulation box, and it achieves optimal $\mathcal O(N)$ complexity and strong scalability, thanks to the random batch Ewald (RBE) approach. Moreover, the additional cost for polarization contributions, represented as image reflection series, is reduced to a negligible cost via combining the RBE with an efficient structure factor coefficient re-calibration technique in $\V k$-space.
Explicit formulas for optimal parameter choices of the algorithm are provided through error estimates, together with a rigorous proof.
Finally, we demonstrate the accuracy, efficiency and scalability of our method, called RBE2D, via numerical tests across a variety of prototype systems. 
An excellent agreement between the RBE2D and the PPPM method is observed, with a significant reduction in the computational cost and strong scalability, demonstrating  that it is a promising method for a broad range of charged systems under quasi-2D confinement.

\end{abstract}

\section{Introduction}

Quasi two-dimensional (quasi-2D) Coulomb systems~\cite{mazars2011long} have caught much attention in many areas of science and engineering. Due to the confinement effect, such systems can exhibit various interesting properties for future nanotechnologies, prototype examples include graphene~\cite{novoselov2004electric}, metal dichalcogenide monolayers~\cite{kumar2012tunable}, and colloidal monolayers~\cite{mangold2003phase}.
Among quasi-2D systems, charged systems under dielectric slab confinement are of particular interest.
The dielectric confinement is relevant to a wide range of systems and phenomena, including water/electrolytes/ionic liquids confined in thin films~\cite{raviv2001fluidity}, ion transportation in nanopores~\cite{zhu2019ion}, and self-assembly of colloidal/polymer monolayers~\cite{kim2017imaging}.
For particle-based simulations of isotropic Coulomb systems, deterministic linear-scaling algorithms have been developed, which usually fall into one of the two categories: fast multipole methods (FMM)~\cite{greengard1987fast,cheng1999fast,ying2004kernel} and fast Fourier transform (FFT) based Ewald-splitting methods~\cite{hockney2021computer,darden1993particle,essmann1995smooth}.
Noteworthy is that a new alternative, namely the random batch Ewald (RBE) method~\cite{jin2021random}, has been proposed recently. The RBE adopts an importance sampling strategy in $\V k$-space, and scales optimally as $\mathcal O(N)$ with reduced variances. The RBE has been successfully applied to large-scale simulations of Coulomb systems under various ensembles~\cite{liang2022superscalability,liang2024JCP}, but so far still limited to fully periodic cases.
For confined quasi-2D Coulomb systems, the joint impact of the long-range Coulomb interaction, sharp dielectric interface conditions, and the reduced symmetry of quasi-2D geometry, poses significant challenges for efficient and accurate particle-based simulations.

Over the past decades, numerous methods have been developed in the literature for quasi-2D Coulomb systems, mostly rely on modifications to the FMM and FFT-based Ewald methods. 
One of the pioneer work is the Ewald2D method~\cite{parry1975electrostatic}, which directly applies the Ewald-splitting technique for quasi-2D systems, it scales as $\mathcal{O}(N^2)$ and is found to be slowly convergent due to the oscillatory nature of the Green's function for quasi-2D geometry.
To further accelerate the convergence of the quasi-2D lattice summation, more methods have been proposed, including the MMM2D method~\cite{arnold2002mmm2d}, the SOEwald2D method~\cite{gan2024fast}, the periodic FMM method~\cite{yan2018flexibly}, and the spectral Ewald method~\cite{lindbo2012fast}.
Instead of exactly solving the quasi-2D problem, alternative approaches are to approximate the quasi-2D summations by 3D periodic summations with correction terms, such as the Yeh--Berkowitz (YB) correction~\cite{yeh1999ewald} and the electric layer correction (ELC)~\cite{arnold2002electrostatics}.
These correction methods effectively reduce the computational cost and are simple to implement, but lacks control of accuracy.

The aforementioned methods face limitations when dealing with sharp dielectric interfaces. To tackle simulations involving dielectrically confined quasi-2D systems, several approaches have emerged in recent years. These approaches address the polarization effect either by introducing image charges for slab interfaces, effectively transforming the system into a homogeneous one (albeit with a considerable increase in thickness along the $z$-direction)~\cite{tyagi2007icmmm2d,tyagi2008electrostatic,dos2015electrolytes,yuan2021particle,liang2020harmonic}, or by numerically solving the Poisson equation with interface conditions~\cite{maxian2021fast, nguyen2019incorporating, ma2021modified}. 
We will not delve into a comprehensive review and comparison of these methods; instead, we will highlight two crucial aspects: 1) they all rely on either FMM or FFT to achieve linear scaling, and 2) the computational cost significantly increases compared to the homogeneous case, especially for \emph{strongly confined} systems (i.e., with a large aspect ratio of simulation box). 
In this scenario, FFT-based methods require a significant increase in the number of grids to accommodate for the extended system thickness (i.e. zero-padding), while FMM-based methods require incorporating more near-field contributions and solving a possibly ill-conditioned linear system~\cite{pei2023fast}. 




In this paper, we propose a novel method, namely the random batch Ewald2D method (RBE2D), for efficient and accurate simulations of quasi-2D charged systems under dielectric confinement.
We first present the reformulation for the quasi-2D Ewald sum into a Ewald3D sum with an ELC term and an infinite boundary correction (IBC) term. Rigorous error analysis is also presented, justifying the optimal parameter choices for a given accuracy.
Then RBE is applied to achieve $\mathcal{O}(N)$ complexity: the stochastic approximation has reduced variances thanks to an importance sampling strategy and coupling with a proper thermostat in MD simulations.
For systems with dielectric mismatch, the subtle dielectric interface problem is reduced to a homogeneous problem via image charge reflection,
we further recalibrate the structure factor coefficients for the image series in $\V k$-space, 
enabling the computation of the polarization contributions with minimal overhead. The main advantage of our method comparing with existing methods is two folds: 1) our method relies on a random batch sampling strategy in $\V k$-space to achieve linear-scaling, thus it is highly efficient and has strong scalability; 2) our method is mesh-free, and can be applied to strongly confined systems with negligible extra cost.
Finally, numerical tests are presented across a variety of prototype systems, including symmetric/asymmetric electrolytes near neutral/non-neutral slabs in an implicit solvent, and all-atom simulations for confined water molecules, validating the accuracy, efficiency, and scalability of the RBE2D method.

The rest of the paper is organized as follows. 
In Section~\ref{sec::RBE2D}, a comprehensive review of the Ewald2D summation is provided, then the RBE2D method is discussed in details.
In Section~\ref{sec::RBE2D_dielectric}, we introduce the fast algorithms for treating the dielectric interfaces, and extend the RBE2D method to systems under dielectric confinement.
In Sections~\ref{sec::numerical} and \ref{sec::numericalDielectric}, the accuracy, efficiency, and scalability of the RBE2D method are demonstrated via various numerical tests. 
Finally, the paper is concluded in Section~\ref{sec::conclusion}.

\section{Random batch Ewald method for quasi-2D Coulomb systems}\label{sec::RBE2D}

In this section, we first provide a comprehensive review of the commonly used summation formulas for quasi-2D Coulomb systems.
Then we discuss its reformulation, and the RBE2D method along with detailed error analysis.

\subsection{Quasi-2D lattice summations}\label{subsec::quasi-2D}

Quasi-2D systems are usually modeled as a rectangular domain $\Omega_{\text{Q2D}}$, with doubly-periodic boundary conditions in~$xy$ and free-space boundary condition in~$z$.
Side lengths of the simulation cell are $L_x$, $L_y$ and $H$ in~$x$, $y$ and $z$, respectively.
The confined particles are modelled as point charges inside $\Omega_{\text{Q2D}}$, with charge~$q_i$ and location~$\V{r}_i$, for~$i\in\{1,\cdots,N\}$. Also see Fig.~\ref{fig:2dIllustration}(a) for a 3D illustration. Typically, we assume $H<L_x=L_y$ due to confinement. 
We first focus on the homogeneous case in this section, and the extension to dielectric confinement case, as illustrated in Fig.~\ref{fig:2dIllustration}(b), will be discussed in Section~\ref{sec::RBE2D_dielectric}.

\begin{figure}[ht]
	\begin{center}
		\includegraphics[width=0.75\textwidth]{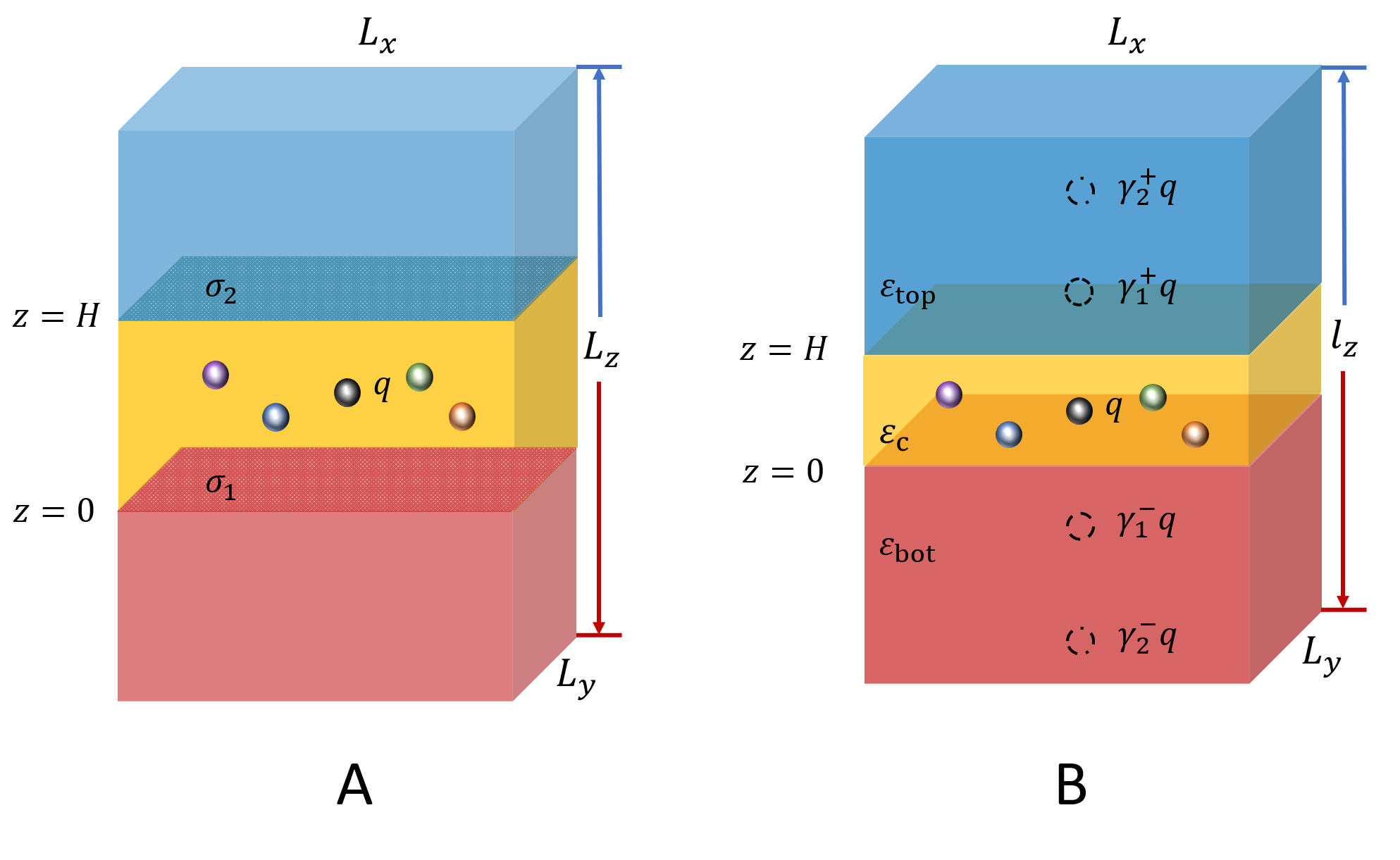}
		\caption{3D illustrations for (a) the quasi-2D system confined by two slabs without dielectric interfaces; and (b)  with two dielectric interfaces. The slabs are located at $z=0$ and $z=H$. 
			In (b), the dashed circles indicate the image charges of $q$ resulting from the reflection between the two interfaces. The magnitude of the $l$th layer of images is altered by factors of $\gamma^{(l)}_{\pm}$.}
		\label{fig:2dIllustration}
	\end{center} 
\end{figure}

The total electrostatic energy $U$ reads:
\begin{equation}\label{eq:2dperiodicU}
	U = \frac{1}{2}\sum_{\V n}{}^\prime\sum_{i, j=1}^N q_iq_j\frac{1}{\abs{\V r_{ij}+ \V \ell}} \;,
\end{equation}
where $\V n=\begin{pmatrix} n_x, n_y\end{pmatrix}$ with $n_x$, $n_y\in\mathbb Z$, and $\V \ell= \begin{pmatrix} n_xL_x, n_yL_y, 0\end{pmatrix}$, i.e., we sum over periodic replicas in $xy$ directions.
It is understood that the singular term when $i=j$ and $\V n=\V 0$ needs to be avoided, which is denoted as $\sum{}^\prime$.
Furthermore, the system must satisfy the total charge neutrality condition
\begin{equation}\label{eq:neutrality}
	\sum_{i=1}^N q_i=0\;,
\end{equation}
so that the electrostatic energy and forces are well-defined.
The difficulty of calculating Eq.~\eqref{eq:2dperiodicU} is two folds, 
1) due to the long-range nature of the Coulomb potential, the lattice sum is slowly convergent, and one can not just truncate the sum at certain finite cutoff distance; 
2) besides the long-range part, the Coulomb potential is also singular as $\V r\to \V 0$, as a result one can not simply deal with it via Fourier transform due to the non-smoothness.

To overcome these difficulties, one can use the so-called ``Ewald-splitting'' technique~\cite{ewald1921berechnung} to decompose the Coulomb kernels into the sum of two components as
\begin{equation}\label{eq:ewald_decomposition}
	\begin{aligned}
		\frac{1}{r}=\frac{\mathrm{erf}(\alpha r)}{r}+\frac{\mathrm{erfc}(\alpha r)}{r} \;
	\end{aligned}
\end{equation}
for any constant $\alpha>0$, where the error function $\mathrm{erf}(\cdot)$ is defined as
\begin{equation}\label{eq:erf}
	\begin{aligned}
		\mathrm{erf}(x):=\frac{2}{\sqrt{\pi}}\int_0^x e^{-u^2}du\;,
	\end{aligned}
\end{equation}
and the error complementary function is $\mathrm{erfc}(x):=1-\mathrm{erf}(x)$.
After splitting, the singular and long-range Coulomb kernel is separated into a singular short-range kernel and a smooth long-range kernel, which can be efficiently handled in real and reciprocal spaces, respectively.
It has been shown that under the charge neutrality condition Eq.~\eqref{eq:neutrality}, the quasi-2D lattice sum $U$ given by Eq.~\eqref{eq:2dperiodicU} can be written as:
\begin{equation}
	U = U_{\text{real}}+U_{\text{Fourier}}\;,
\end{equation}
where
\begin{equation}\label{eq:ewald2d-1}
	U_{\te{real}}=\frac{1}{2}\sum_{\V n}{}^\prime\sum_{i, j=1}^N q_iq_j\frac{\mathrm{erfc}(\alpha\abs{\V r_{ij}+\V \ell })}{\abs{\V r_{ij}+ \V \ell}} \;,
\end{equation}
\begin{equation}\label{eq:ewald2d-2}
	U_{\te{Fourier}}=\frac{\pi}{2L_xL_y}\sum_{i, j=1}^N q_iq_j\sum_{\V h}{}^\prime\frac{e^{\i \V h\cdot \V r_{ij}}}{h}\mathcal{G}_{\alpha}(h,z_{ij})-\frac{\alpha}{\sqrt{\pi}}\sum_{i=1}^{N}q_i^2+\mathcal J_0\;.
\end{equation}
Here the reciprocal lattice vector is defined as $ {\V h}=2\pi\begin{pmatrix} n_x/L_x, n_y/L_y, 0\end{pmatrix}$, $\i=\sqrt{-1}$, and
\begin{equation}
	\mathcal{G}_{\alpha}(h,z_{ij}):=\left[e^{h z_{ij}}\mathrm{erfc}\left(\frac{h}{2\alpha}+\alpha z_{ij}\right)+e^{-h z_{ij}}\mathrm{erfc}\left(\frac{h}{2\alpha}-\alpha z_{ij}\right)\right],
\end{equation}
with the $\V 0$-th mode correction $\mathcal J_0$ reads
\begin{equation}\label{eq:ewald2d-j0}
	\begin{aligned}
		\mathcal J_0 &=\frac{-\pi}{L_xL_y}\sum_{i, j=1}^N q_iq_j\left[z_{ij}\mathrm{erf}\left(\alpha z_{ij}\right)+\frac{1}{\alpha\sqrt{\pi}}e^{-\alpha^2z^2_{ij}}\right]\;.
	\end{aligned}
\end{equation}
Eqs.~\eqref{eq:ewald2d-1}--\eqref{eq:ewald2d-j0} are the well-known exact Ewald2D summation formulas~\cite{parry1975electrostatic,zhonghanhu2014JCTC}. 
Unlike the 3D-periodic case where Ewald sum is conditionally convergent, the Ewald2D sum is absolutely convergent.
However, direct computation of~\eqref{eq:ewald2d-1}--\eqref{eq:ewald2d-j0} is quite involved and takes $\mathcal O(N^2)$ complexity, much worse than $\mathcal O(N^{1.5})$ for Ewald3D summations, limiting its application to large-scale simulations.

\subsection{Reformulation for Ewald2D summation}\label{subsec::IBCELC}

In this section, we present a reformulation for the Ewald2D sum,
which reduces the Ewald2D sum (Eqs.~\eqref{eq:ewald2d-1}--\eqref{eq:ewald2d-j0}) to a simpler Ewald3D sum complemented by an infinite boundary correction (IBC) and electrostatic layer corrections (ELC)~\cite{pan2014rigorous}. Based on which efficient algorithms can be further developed.

We first introduce the following identity~\cite{erdelyi1954tables}
\begin{equation}\label{eq:int-identity1}
	\begin{aligned}
		I(\omega, \nu)&=\frac{\pi}{2\omega}\left[e^{\omega\nu}\mathrm{erfc}\left(\omega+\frac{\nu}{2}\right)+e^{-\omega\nu}\mathrm{erfc}\left(\omega-\frac{\nu}{2}\right)\right]=e^{-\omega^2}\int_{-\infty}^{\infty}\frac{e^{-t^2}}{\omega^2+t^2}e^{\i t\nu}dt\;.
	\end{aligned}
\end{equation}
Taking the limit $\omega\to 0$ and removing the divergence, one has the following identity
\begin{equation}\label{eq:int-identity2}
	\begin{aligned}
		I_0(\nu)&=-\pi\left[\nu\mathrm{erf}\left(\frac{\nu}{2}\right)+\frac{2e^{-\frac{\nu^2}{4}}}{\sqrt{\pi}}\right]=\int_{-\infty}^{\infty}\frac{e^{-t^2}e^{\i t\nu}-1}{t^2}dt\;.
	\end{aligned}
\end{equation}
Applying Eqs.~\eqref{eq:int-identity1} and~\eqref{eq:int-identity2} to Eqs.~\eqref{eq:ewald2d-2} and~\eqref{eq:ewald2d-j0} by taking $\omega=h/2\alpha$ and $\nu=2\alpha z_{ij}$, one obtains the alternative expressions in integral forms:
\begin{equation}\label{eq:ewald2d-intform1}
	\begin{aligned}
		U_{\te{Fourier}}&=\frac{1}{2\alpha L_xL_y}
		\sum_{i,j=1}^{N}q_iq_j\sum_{\V h}{}^\prime e^{\i \V h\cdot \V r_{ij}}\int_{-\infty}^{\infty}\frac{e^{-\frac{h^2}{4\alpha^2}-t^2}}{\frac{h^2}{4\alpha^2}+t^2}e^{2\i \alpha z_{ij}t}dt-\frac{\alpha}{\sqrt{\pi}}\sum_{i=1}^{N}q_i^2+\mathcal{J}_0,
	\end{aligned}
\end{equation}
with
\begin{equation}\label{eq::J02}
	\mathcal{J}_0=\frac{1}{2\alpha L_xL_y}\sum_{i,j=1}^{N}q_iq_j\int_{-\infty}^{\infty}\frac{e^{-t^2}e^{2\i \alpha z_{ij}t}-1}{t^2}dt.
\end{equation}
Due to the smoothness and exponential decay of the integrands, discretizing the integrals via trapezoidal rule yields spectral convergence~\cite{trefethen2014Rev}. The resulting discretized formulation along with error estimates, are summarized in Theorem~\ref{thm::err}.

\begin{theorem}\label{thm::err}
	\emph{\textbf{(Ewald2D reformulation and its error estimate)}} Assume $L_z\geq H$, discretizing the integrals in Eq.~\eqref{eq:ewald2d-intform1} via trapezoidal rule with mesh size $\pi/\alpha L_z$, one will have 
	\begin{equation}\label{eq:ewald2d-ELC}
		\begin{aligned}
			U_{\emph{Fourier}}&=\frac{2\pi}{L_xL_yL_z}\sum_{\V k}{}^\prime\frac{e^{-\frac{k^2}{4\alpha^2}}}{k^2}\abs{\sum_{i=1}^N q_i e^{\i \V k\cdot \V r_i}}^2 -\frac{\alpha}{\sqrt{\pi}}\sum_{i=1}^{N}q_i^2+U_{\emph{IBC}}+U_{\emph{ELC}}+U_{\emph{err}}\;,
		\end{aligned}
	\end{equation}
	where $\bm{k}=2\pi(n_x/L_x,n_y/L_y,n_z/L_z)$, the first two terms on the RHS are in the form of Ewald3D summation; 
	and
	\begin{equation}
		U_{\emph{IBC}}:=\frac{2\pi}{L_xL_yL_z}\left(\sum_{i=1}^{N} q_iz_i\right)^2\quad\text{~}\quad U_{\emph{ELC}}:=\frac{2\pi}{L_xL_y}\sum_{i,j}q_iq_j\sum_{\V h}{}^\prime\left[\frac{e^{\i \V h\cdot \V r_{ij}}}{h}\frac{\cosh(hz_{ij})}{1-e^{hL_z}}\right]
	\end{equation}
	are the IBC~\cite{yeh1999ewald} and ELC~\cite{arnold2002electrostatics} terms found earlier separately, and $U_{\emph{err}}$ is the remainder error term dueto trapezoidal discretization. The last two terms decay as
	\begin{equation}\label{eq::ELCerr}
		|U_{\emph{ELC}}|\sim \mathcal{O}\left(e^{-\frac{2\pi(L_z-H)}{\max\{L_x,L_y\}}}\right)\quad\text{and}\quad |U_{\emph{err}}|\sim \mathcal{O}\left(e^{-\alpha^2 (L_z-H)^2}\right),
	\end{equation}
	respectively.
\end{theorem}
\begin{proof}
	Through the error analysis of the trapezoidal rule provided in Appendix~\ref{app::trapezoidal}, one has 
	\begin{equation}\label{eq::17}
		\begin{split}
			&\sum_{\V h}{}^\prime e^{\i \V h\cdot \V r_{ij}}\int_{-\infty}^{\infty}\frac{e^{-\frac{h^2}{4\alpha^2}-t^2}}{\frac{h^2}{4\alpha^2}+t^2}e^{2\i \alpha z_{ij}t}dt+\int_{-\infty}^{\infty}\frac{e^{-t^2}e^{2\i \alpha z_{ij}t}-1}{t^2}dt\\
			=&\frac{\pi}{\alpha L_z}\sum_{\bm{k}}{}^\prime \frac{e^{-\frac{k^2}{4\alpha^2}}}{k^2}e^{\i\bm{k}\cdot\bm{r}_{ij}}+\frac{4\alpha \pi}{L_z}\lim_{k_z\rightarrow 0}\frac{e^{-\frac{k_z^2}{4\alpha^2}}e^{\i k_z z_{ij}}-1}{k_z^2}+4\alpha\pi \sum_{h}{}^{\prime} \frac{e^{\i \bm{h}\cdot \bm{r}_{ij}}}{h}\frac{\cosh(hz_{ij})}{1-e^{hL_z}}\\
			&+\mathcal{O}\left(e^{-\alpha^2(L_z-|z_{ij}|)^2}\right).
		\end{split}
	\end{equation}
	By substituting Eq.~\eqref{eq::17} into Eq.~\eqref{eq:ewald2d-intform1}, it is observed that the first two terms on the RHS of Eq.~\eqref{eq::17} correspond to the first term and the IBC term in Eq.~\eqref{eq:ewald2d-ELC}, respectively. The third term in Eq.~\eqref{eq::17} acts as the ELC term in Eq.~\eqref{eq:ewald2d-ELC}. 
\end{proof}

Clearly, the first two terms in Eq.~\eqref{eq:ewald2d-ELC} closely resemble the Fourier component of the Ewald3D summation~\cite{ewald1921berechnung,zhonghanhu2014JCTC}. It can be understood as constructing a fully periodic system, referred to as $\Omega_{\text{3D}}$, which shares the same particle configuration, and $xy$ dimensions as the original quasi-2D system $\Omega_{\text{Q2D}}$; but extending the dimension in $z$ from $H$ to $L_z$ (see Fig.~\ref{fig:2dIllustration}(a)). This extra layer with thickness of~$L_z-H$ (containing no particles) is usually called a ``vacuum layer''.
Now according to Theorem~\ref{thm::err}, one has
\begin{equation}\label{eq::U_Q2D}
	U_{\text{Fourier}}(\Omega_{\text{Q2D}}) = U_{\text{Fourier}}(\Omega_{\text{3D}}) + U_{\text{IBC}} + U_{\text{ELC}} + U_{\text{err}}\;,
\end{equation}
The IBC and ELC terms can be understood to correct the unwanted $z$-replicas effect and boundary conditions at infinity when approximating the quasi-2D system as a fully periodic one with vacuum layers~\cite{arnold2002electrostatics,yeh1999ewald}.

The $U_{\text{ELC}}$ and $U_{\text{err}}$ terms are usually disregarded by mainstream MD softwares~\cite{thompson2021lammps,ABRAHAM201519}, and the remainning terms,~$U_{\text{Fourier}}(\Omega_{\text{3D}})$ and~$U_{\text{IBC}}$, can be conveniently evaluated via FFT-based Ewald3D methods. 
However, Eq.~\eqref{eq::ELCerr} indicates that to guarantee both the ELC and $U_{\text{err}}$ terms contribute below a certain tolerance $\varepsilon$, the choice for the thickness $L_z$
should satisfy
\begin{equation}
	L_z\geq H+\max\left\{\frac{\max\{L_x,L_y\}}{2\pi}\log\frac{1}{\varepsilon}, ~\frac{1}{\alpha}\sqrt{\log\frac{1}{\epsilon}}\right\}\;,
\end{equation}
which means one may need to choose a vacuum layer thickness much larger than~$H$ if $\alpha \ll 1/H$ or $\max\{L_x,L_y\}\gg H$. 
For FFT-based methods, this means one will need to discretize the whole extended domain with vacuum layers, which will greatly increase the computational cost.

\begin{remark}
	It is interesting to observe that the thickness of the vacuum layer ($L_z-H$) is determined by the system dimensions in $xy$ and the Ewald splitting parameter $\alpha$, but is independent of $H$.
\end{remark}

\subsection{Random batch Ewald2D method}

We now introduce the random batch Ewald2D (RBE2D) method.
Instead of deterministically compute the $\V k$-space sum either directly or via FFT, we now evaluate it through a stochastic approximation, namely, the random ``mini-batch'' sampling~\cite{jin2020random}. 
Notice that the low $\V k$ modes are more important due to the $e^{-k^2/(4\alpha^2)}$ factor, an importance sampling strategy was further developed for variations reduction.
The random batch Ewald (RBE) method has recently been developed for fully periodic systems~\cite{jin2021random,liang2022superscalability, liang2023random}, and here we extend it to quasi-2D systems.

The long-range component of the electrostatic force exerts on $i$-th particle is given by $\V F_i^{\te{Fourier}}=-\nabla_{\bm{r}_i}U_{\text{Fourier}}$, its random batch approximation is denoted as
\begin{equation}\label{eq:forceRBE}
	\begin{aligned}
		\V {\tilde{F}}_i^{\te{Fourier}}=-\sum_{\ell =1}^P \frac{S}{P}\frac{4\pi q_i\V k_\ell}{Vk_\ell^2}e^{-k_\ell^2/(4\alpha^2)}\te{Im}\left(e^{-\i \V k_\ell\cdot \V r_i}\rho(\V k_\ell)\right)\;,
	\end{aligned}
\end{equation}
where $P$ is the batch size, $V=L_xL_yL_z$ is the volume of the extended simulation domain $\Omega_{\text{3D}}$, and $S$ is the normalizing constant,
\begin{equation}\label{eq:normalizingS}
	\begin{aligned}
		S=\sum_{\V k \neq 0} e^{-k^2/(4\alpha^2)}\;.
	\end{aligned}
\end{equation}
In MD simulations, it is convenient to employ the Metropolis algorithm~\cite{metropolis1953equation} to sample from $\mathscr{P}(\bm{k})=e^{-k^2/(4\alpha^2)}/S$.

For quasi-2D systems, the RBE2D method offers a particular advantage: it is mesh free, so that unlike existing grid-based fast algorithms, the computational cost is independent with added vacuum layer thickness. 
This statement is justified more rigorously in Theorem \ref{Thm::1}.

\begin{theorem}\label{Thm::1} 
	Denote the fluctuation of the random batch approximation in force by $\V \Xi_{\V F,i}=\V {\tilde{F}}_i^{\emph{Fourier}}-\V {F}_i^{\emph{Fourier}}$. The random variable $\V \Xi_{\V F,i}$ has zero expectation, i.e., $\mathbb{E}\V \Xi_{\V F,i}=\V 0$, and its variance is given by
	\begin{equation}\label{eq::30}
		\mathbb{E}|\V \Xi_{\V F,i}|^2=\frac{1}{P}\left[\frac{(4\pi q_i)^2S}{V^2}\sum_{\bm{k}\neq \bm{0}}\frac{e^{-k^2/(4\alpha^2)}}{k^2}\left|\emph{Im}\left(e^{-\i \V k_\ell\cdot \V r_i}\rho(\V k_\ell)\right)\right|^2-|\V F_{i}^{\emph{Fourier}}|^2\right]
	\end{equation}
	which is bounded and scales as $\mathcal O(1/P)$. Furthermore, under the  Debye-H\"uckel (DH) assumption~\cite{debye1923theorie,hansen2013theory}, $\mathbb{E}|\V \Xi_{\V F,i}|^2$ is independent of both particle number $N$ and size $L_z$.
\end{theorem}

\begin{proof}
	The property of unbiasedness in Theorem~\ref{Thm::1} is straightforward and guarantees the consistency of the stochastic approximation, i.e., $\mathbb{E}\V{ \tilde{F}}_i^{\te{Fourier}}=\mathbb{E}\V {F}_i^{\te{Fourier}}$. 
	Let us now consider the variance. By the DH theory, the structure factor term in Eq.~\eqref{eq::30} is bounded by a constant $C$~\cite{jin2021random} for $\bm{k}\neq \bm{0}$,
	\begin{equation}
		\left|\te{Im}\left(e^{-\i \V k\cdot \V r_i}\rho(\V k)\right)\right|^2\leq C,
	\end{equation}
	and vanishes for $\bm{k}=\bm{0}$. 
	Further by the monotonicity of Gaussian on $[0, +\infty ]$, it follows that
	\begin{equation}\label{eq::32}
		\begin{split}
			\mathbb{E}|\V \Xi_{\V F,i}|^2\leq \frac{8 CSq_i^2}{PV}\int_{0}^{\infty}e^{-k^2/(4\alpha^2)}dk= \frac{8\sqrt{\pi}\alpha CSq_i^2}{PV}\;.
		\end{split}
	\end{equation}
	Analogously, one can obtain the following estimate for the normalization constant $S$:
	\begin{equation}\label{eq::Sa}
		\begin{split}
			S\leq \frac{V}{(2\pi)^3}\int_{0}^{\infty}4\pi k^2e^{-k^2/(4\alpha^2)}dk=\frac{\alpha^3V}{\pi^{3/2}}.
		\end{split}
	\end{equation}
	Substituting Eq.~\eqref{eq::Sa} into Eq.~\eqref{eq::32} yields
	\begin{equation}
		\mathbb{E}|\V \Xi_{\V F,i}|^2\leq\frac{8\alpha^4Cq_i^2}{\pi P}\sim \mathcal O\left(\frac{1}{P}\right)
	\end{equation}
	which is independent of either $N$ or $L_z$.
\end{proof}

Theorem~\ref{Thm::1} suggests that the variance in the random batch approximation can be effectively controlled 
by appropriately choosing the batch size $P$, which is independent with particle number $N$.
In the context of Ewald summation, typical choices for $\alpha$ are $(\sqrt{N}/V)^{1/3}$ for direct truncation~\cite{kolafa1992cutoff} and $(N/V)^{1/3}$ for FFT-based calculation~\cite{deserno1998mesh}. 
Now for the RBE2D method, it is clear that changing the vacuum layer thickness does not affect its efficiency. 
In practice, we set $\alpha\sim N^{1/3}/(L_xL_yH)^{1/3}$ (independent of $L_z$), then it is expected that the variance of the force remains at the same level as one changes $L_z$.
This is further validated by numerical results in Section~\ref{subsec::electrolyte-neutral}.

\subsection{Convergence and complexity of the RBE2D} \label{sec::convergence}

In this section, we further examine the convergence and complexity of RBE2D accelerated MD simulations. 
First, we revisit the the common purpose of MD simulations, which is to explore the configuration space and obtain static/dynamic ensemble-averaged properties by solving the equations of motion for $N$ interacting particles.
For example, consider the widely used canonical ensemble, where the system has fixed particle number $N$, volume $V$ and temperature $T$, which can be simulated as solving the following equations, namely, MD with Langevin thermostat~\cite{frenkel2023understanding}:
\begin{equation}\label{Langevin}
	\begin{aligned}
		& d \bm{r}_i=\bm{v}_i d t, \\
		& m_i d \bm{v}_i=\left[\bm{F}_i-\gamma \bm{v}_i\right] d t+\sqrt{2 \gamma k_{\mathrm{B}} T} d \bm{W}_i,
	\end{aligned}
\end{equation}
where $\bm{r}_i$, $m_i$, and $\bm{v}_i$ represent the position, mass, and velocity of the $i$-th particle, respectively. $\V{F}_i$ is the force exerts on the $i$-th particle, $\{\bm{W}_i\}$ are i.i.d.~Wiener processes, and $\gamma$ is the reciprocal characteristic time scale associated with the thermostat.
In practice, the stochastic differential equations are discretized with proper numerical schemes and integrated with time step~$\Delta t$.
Let~$(\V{r}_i, \V{v}_i)$ be the solution to Eq.~\eqref{Langevin} with~$\bm{F}_i = \bm{F}^{\text{exact}}_i$, where~$\bm{F}^{\text{exact}}_i$ is the exact force on the $i$-th particle; and let $(\widetilde{\bm{r}}_i,\widetilde{\bm{v}}_i)$ be the solution to the same set of equations but with RBE2D approximated force given by~$\bm{F}_i = \bm{F}^{\text{exact}}_i + \bm{\Xi}_{\bm{F},i}$.
Then the relation between the two sets of solutions satisfies the following Theorem~\ref{thm::2}~\cite{jin2021convergence}.

\begin{theorem}\label{thm::2}
	If the forces $\bm{F}_i$ are bounded and Lipschitz, and the fluctuation of stochastic force satisfies $\mathbb{E}\V \Xi_{\V F,i}=0$ and $\mathbb{E}|\V \Xi_{\V F,i}|^2\leq \Lambda$, then for any $t_{\emph{MD}}>0$, there exists $C(t_{\emph{MD}})>0$ such that
	\begin{equation}\label{eq::Convergence}
		\sup _{t \in[0, t_{\emph{MD}}]} \sqrt{\frac{1}{N} \sum_{i=1}^N \mathbb{E}\left(\left|\bm{r}_i-\widetilde{\bm{r}}_i\right|^2+\left|\bm{v}_i-\widetilde{\bm{v}}_i\right|^2\right)} \leq C(t_{\emph{MD}}) \sqrt{\Lambda \Delta t}.
	\end{equation}
\end{theorem}
Theorem~\ref{thm::2} indicates that the RBE2D approximated dynamics is capable of capturing finite time structure and dynamic properties.
For long-time simulations, additional assumptions are needed regarding force regularity~\cite{jin2022random}. 
It worths noting that, although the Coulomb potential is neither bounded nor Lipschitz at $r=0$, actual MD simulations will include the Lennard-Jones (LJ) potential, which provides a stronger repulsive force when particles get close, thus avoiding the divergence issue in $\bm{F}_i$. So that Theorem \ref{thm::2} can still valid for practical MD simulations.

Additionally, following Theorem \ref{thm::2}, the fluctuations introduced by random batch approximation, $\V \Xi_{\V F,i}$,  can be physically understood as heating effects, which will cause a temperature drift for the system and is unwanted. 
For NVT and NPT ensembles, the heating effect can be eliminated by introducing appropriate thermostats such as the Langevin thermostat~\cite{feller1995constant}, the Nos\'e-Hoover (NH) thermostat~\cite{hoover1985canonical}, etc. With an additional weak-coupled bath on the Newtonian dynamics to avoid energy drift, the NVE ensemble can also be accurately simulated with random batch-type approximations~\cite{liang2024JCP}.


Next, we analyze the computational complexity of the RBE2D accelerated MD simulations.
The detailed algorithm is summarized in Algorithm~\ref{Alg::1}.
For a system with~$N$ particles, we can set~$\alpha \sim N^{1/3}/(L_xL_yH)^{1/3}$, then the real space computation complexity is of~$\mathcal{O}(N)$; and by the random batch importance sampling strategy, the cost of the $\V k$ space computation is of~$\mathcal{O}(PN)$, where~$P\sim \mathcal O(1)$ the batch size which is independent of $N$.
As such, the total computational complexity is of~$\mathcal{O}(N)$.

\begin{algorithm}[ht]
	\caption{(RBE2D accelerated molecular dynamics for quasi-2D systems)}\label{Alg::1}
	\begin{algorithmic}[1]
		\State Choose $\alpha$, $r_c$ (the cutoff in real space), $\Delta t$, the size of $L_z$, and batch size $P$. Initialize the positions and velocities of ions as well as surface charge densities. Set $N_T$ as the total MD time steps.
		\For {$n \text{ in } 1: N_T$}
		\State Sample $P$ nonzero frequencies $\bm{k}\sim e^{-k^2/(4\alpha)}$ by the Metropolis procedure to form  set $\mathcal{K}$.
		\State The $\V k$-space component of the force, $\bm{F}_i^{\te{Fourier}}$, is computed using the random batch approximation $\widetilde{\V F}_i^{\te{Fourier}}$ with the $P$ frequencies chosen from $\mathcal{K}$.
		\State The real-space component of the force, $\bm{F}_i^{\te{real}}$, is computed by the neighbor list method. 
		\State Integrate the equations of motion for time step $\Delta t$ with an appropriate thermostat and its corresponding integration scheme. 
		
		\EndFor
	\end{algorithmic}
\end{algorithm}
Besides its linear complexity, the RBE2D method offers two other significant advantages:
\begin{itemize}
	\item Communication efficiency: thanks to the sampling nature, RBE2D only requires a communication cost of~$O(1)$. This streamlined communication facilitates easy parallelization and high scalability.
	\item Vacuum layer flexibility: the RBE2D is mesh free, and choosing a larger vacuum layer thickness incurs no extra cost. This provides flexibility for one to always choose a proper $L_z$ to guarantee accuracy, without sacrificing efficiency.
\end{itemize}

\section{Quasi-2D systems under dielectric confinement}\label{sec::RBE2D_dielectric}
In this section, we extend the RBE2D to the challenging case of systems with dielectric jumps at the slab interfaces ($z=0$ and $z=H$).
We will first introduce the three-layer sandwich model, and its computation by truncated image reflection series.
Detailed analysis of the truncation error is presented; along with a recalibrated efficient formula for the computation of electrostatic energy and force with minimum extra cost. Finally, we summarize the algorithm framework to extend the RBE2D method to systems under dielectric confinement.

\subsection{The three-layer sandwich model}
A schematic of the three-layer sandwich model illustrating the dielectrically confined quasi-2D system is shown in Fig.~\ref{fig:2dIllustration}(b). 
Two dielectric interfaces are introduced at $z=0$ and $z=H$, dividing the system into three layers. Each of these layers, from top to bottom, are labeled as $\Omega_{\te{top}}$, $\Omega_{\te{c}}$, and $\Omega_{\te{bot}}$, with respective permittivities $\varepsilon_{\te{top}}$, $\varepsilon_{\te{c}}$, and $\varepsilon_{\te{bot}}$. 
Additionally, we still assume the total charge neutrality condition (Eq.~\eqref{eq:neutrality}) is valid. 

First, let us define the dimensionless dielectric contrasts at the interfaces as:
\begin{equation}
	\gamma_{\text{top}}=\frac{\varepsilon_{\text{c}}-\varepsilon_{\text{top}}}{\varepsilon_{\text{c}}+\varepsilon_{\text{top}}} \quad \text{and} \quad \gamma_{\text{bot}}=\frac{\varepsilon_{\text{c}}-\varepsilon_{\text{bot}}}{\varepsilon_{\text{c}}+\varepsilon_{\text{bot}}}\;.
\end{equation}
Note that $|\gamma_{\text{top}}|\leq 1$ and $|\gamma_{\text{bot}}|\leq 1$ since $\varepsilon_{\text{c}}>0$. 
We are particularly interested in the electrostatic energy $U^{\text{c}}$, for charges confined within the central layer $[0,H]$. 
The polarization contribution can be conveniently expressed as a series of image charges resulting from 1) periodic replications along $xy$ by doubly-periodicity and 2) mirror reflections along $z$ (refer to \cite{jackson1999classical} and also Fig.~\ref{fig:2dIllustration}(b)):
\begin{equation}\label{eq::Uc}
	U^{\text{c}}=\frac{1}{2}\sum_{i,j=1}^{N}q_iq_j\left[\sum_{\bm{n}}{}^\prime \frac{1}{\left|\bm{r}_{ij}+\bm{\ell}\right|}+\sum_{\bm{n}}\sum_{l=1}^{\infty}\left(\frac{\gamma^{(l)}_{+}}{\left|\bm{r}_i-\bm{r}_{j+}^{(l)}+\bm{\ell}\right|}+\frac{\gamma^{(l)}_{-}}{\left|\bm{r}_i-\bm{r}_{j-}^{(l)}+\bm{\ell}\right|}\right)\right],
\end{equation}
where the factors for the $l$th-order image charges are $\gamma^{(l)}_+=\gamma_{\text{top}}^{\lfloor l/2 \rfloor}\gamma_{\text{bot}}^{\lceil l/2 \rceil}$ and $\gamma^{(l)}_-=\gamma_{\text{top}}^{\lceil l/2 \rceil}\gamma_{\text{bot}}^{\lfloor l/2 \rfloor}$. Here $+(-)$ corresponds to the image charges in $\Omega_{\te{top}}$ ($\Omega_{\te{bot}}$). The notation $\lceil x\rceil$ $(\lfloor x\rfloor)$ represents the ``ceil'' (``floor'') function, and the positions of the $l$th-order image of charge $q_j$ are given by $\bm{r}_{j\pm}^{(l)}=(x_{j},y_j,z_{j\pm}^{(l)})$ with
\begin{equation}\label{eq::31}
	\begin{split}
		z_{j+}^{(l)}=(-1)^lz_{j}+ 2\lceil l/2\rceil H \quad\,\text{and}\quad\, z_{j-}^{(l)}=(-1)^lz_{j}- 2\lfloor l/2\rfloor H.
	\end{split}
\end{equation}
Clearly, compared to the previous homogeneous case, the extra polarization contribution in terms of image charge series can significantly impact the efficiency. This effect becomes more pronounced when $|\gamma_{\text{top}}\gamma_{\text{bot}}|\to 1$,
as the convergence of the image series in Eq.~\eqref{eq::Uc} can become extremely slow.

To make the problem computationally feasible, existing methods~\cite{dos2015electrolytes,dosSantos2016JPCB,yuan2021particle} typically involve truncating the number of image charges in $z$ at certain reflection level $l=M$, and then correspondingly adjust the vacuum layer thickness by choosing $l_z=(2M+1)H$ so as to enclose all the $M$-level reflected images. 
Subsequently, an additional vacuum layer as a safe zone is introduced, increasing the system size in $z$ to $L_z=\eta_{z}l_z$ with $\eta_z> 1$. 
Unfortunately, so far there is no clear theoretical guidance about how to select parameters $M$ and $\eta_z$, with some given accuracy requirement. 
Even so, due to the simplicity of image charge representation, this approach has still been widely adopted in practical simulations to account for the dielectric confinement effect.

In subsequent sections, we will investigate the choice of $M$ and $L_z$ based on the Ewald2D reformulation introduced in section~\ref{subsec::IBCELC}. 
To the best of our knowledge, this work possibly provides the first theoretical study on how $M$ and $L_z$ influence the error. 
It worths noting that, regarding the parameter choice, conflicting conclusions were made in literature based on numerical tests. 
For example, when $|\gamma_{\text{top}}\gamma_{\text{bot}}|\approx 1$, dos Santos and Levin~\cite{dos2015electrolytes} found that approximately $M=50$ levels of images are required to achieve satisfactory accuracy; while Yuan, Antila, and Luijten~\cite{yuan2021particle} suggests that choosing $M=5$ is sufficient. 
In~\cite{yuan2021particle}, it was also observed that choosing a very large $M$ may lead to counter-intuitive accuracy loss. 
Interestingly, our theoretical analysis below can provide explanations for all the subtle phenomena above.

\subsection{Truncation error analysis for the image reflection series}\label{subsec::Ewaldimage}
To analyze the truncation error for the image series in Eq.~\eqref{eq::Uc} at certain reflection level  $l=M$,
it is useful to first introduce the quasi-2D Poisson summation formula, summarized in Lemma \ref{thm:psf}
\begin{lemma}\label{thm:psf} 
	Let $f(\bm{\rho},z)$ be a function which is doubly-periodic in $\bm{\rho}=(x,y)$. Suppose that $f$ has its Fourier transform $\widetilde{f}$ and $\bm{r}=(\bm{\rho},z)$, then one has
	\begin{equation}
		\sum_{\boldsymbol{n}} f(\boldsymbol{r}+\boldsymbol{\ell})=\frac{1}{2 \pi L_x L_y} \sum_{\bm{k}_{\bm{\rho}}} \int_{\mathbb{R}} \tilde{f}(\bm{k}_{\bm{\rho}}, k_z) e^{\i \bm{k}_{\bm{\rho}} \cdot \boldsymbol{\rho}} e^{\i k_z z} d k_z,
	\end{equation}
	where $\bm{k}_{\bm{\rho}}=2\pi(n_x/L_x,n_y/L_y)$.
\end{lemma}
Applying Lemma~\ref{thm:psf} to Eq.~\eqref{eq::Uc}, one has
\begin{equation}\label{eq::33}
	\begin{split}
		U^{\text{c}}&=\frac{1}{2}\sum_{i,j=1}^{N}q_iq_j\sum_{\bm{n}}{}^\prime \frac{1}{\left|\bm{r}_{ij}+\bm{\ell}\right|}+\frac{\pi}{L_xL_y}\sum_{i,j=1}^{N}q_iq_j\sum_{\bm{h}}{}^{\prime}\frac{e^{-\i \bm{h}\cdot\bm{r}_{ij}}}{h}\beta_{ij}(h)\;,
	\end{split}
\end{equation}
where we leave the first term same as Eq.~\eqref{eq:2dperiodicU}, which is the energy for a homogeneous quasi-2D Coulomb system, and 
\begin{align}\label{eq::34}
	\beta_{ij}(h):=\sum\limits_{l=1}^{\infty}\left[\gamma_+^{(l)}e^{-h\left|z_i-z_{j+}^{(l)}\right|}+\gamma_-^{(l)}e^{-h\left|z_i-z_{j-}^{(l)}\right|}\right].
\end{align}
Substituting the definitions of $\gamma_{\pm}^{(l)}$ and Eq.~\eqref{eq::31} into Eq.~\eqref{eq::34} and rearranging terms, we obtain
\begin{equation}
	\begin{split}
		\beta_{ij}(h)&=\left[\gamma_{\text{bot}}e^{-h|z_i+z_{j}|}+\gamma_{\text{top}}e^{-h|2H-(z_i+z_{j})|}+\gamma_{\text{top}}\gamma_{\text{bot}}\left(e^{-h|2H-z_{ij}|}+e^{-h|2H+z_{ij}|}\right)\right]\zeta(h)\;,
	\end{split}
\end{equation}
where 
\begin{equation}\label{eq::36}
	\zeta(h):=\sum_{n=0}^{\infty}\left(\gamma_{\text{top}}\gamma_{\text{bot}}e^{-2hH}\right)^{n}=\frac{1}{1-\gamma_{\text{top}}\gamma_{\text{bot}}e^{-2hH}}\;.
\end{equation}
Finally, we arrive at Theorem \ref{thm::imagetrun} for truncation error of the image reflection series, which can be directly obtained by Eqs.~\eqref{eq::33}-\eqref{eq::36}.
\begin{theorem}\label{thm::imagetrun} 
	If one truncates the image reflection series in Eq.~\eqref{eq::Uc} at $l=M$, the resulting $U^{\emph{c}}_{M}$ is an approximation of $U^{\emph{c}}$ with truncation error
	\begin{equation}\label{eq::37}
		|U^{\emph{c}}-U^{\emph{c}}_{M}|\sim \mathcal{O}\left(\left|\gamma_{\emph{top}}\right|^{\lfloor\frac{M+1}{2}\rfloor}\left|\gamma_{\emph{bot}}\right|^{\lceil\frac{M+1}{2}\rceil}e^{-\frac{2\pi H (M-1)}{\max\{L_x,L_y\}}}\right).
	\end{equation}
\end{theorem}
Theorem~\ref{thm::imagetrun} provides two key insights: 1). when $|\gamma_{\text{top}}\gamma_{\text{bot}}|\ll 1$, the truncation error decreases exponentially in $M$ as $|\gamma_{\text{top}}\gamma_{\text{bot}}|^{\frac{M+1}{2}}$; 2).~the truncation error is significantly influenced by the aspect ratio of the system, i.e., $H/\max\{L_x,L_y\}$. For example, if $H\approx \max\{L_x,L_y\}$ and $|\gamma_{\text{top}}|=|\gamma_{\text{bot}}|=1$, achieving an accuracy of $10^{-5}$ generally needs $M=3$ levels of image reflections. However, if $H\ll \max\{L_x,L_y\}$, a much larger $M$ must be chosen to guarantee the same accuracy. 

In general, the number of reflection level $M$ required to achieve a desired tolerance $\varepsilon$ can be estimated by:
\begin{equation}\label{eq::38}
	M\sim \frac{2\log \varepsilon-\frac{4\pi H}{\max\{L_x,L_y\}}-\log|\gamma_{\text{top}}\gamma_{\text{bot}}|}{\log|\gamma_{\text{top}}\gamma_{\text{bot}}|-\frac{4\pi H}{\max\{L_x,L_y\}}}\;,
\end{equation}
which is a directly consequence of Eq.~\eqref{eq::37}. Eq.~\eqref{eq::38} also explains the conflicting conclusions made in literature. In~\cite{dos2015electrolytes}, a very large aspect ratio is chosen, so that $M=50$ is needed; while in~\cite{yuan2021particle}, $H=\max\{L_x,L_y\}/3$, where choosing $M= 5$ can already provide 5 digits accuracy according to Eq.~\eqref{eq::38}.  

\subsection{Efficient summation formula for dielectrically confined quasi-2D systems}
\label{subsec::IBC&ELC&ICM}

After truncating the image reflection series at $l=M$, the resulting system,demoted as  $\Omega_{\text{Q2D}}^{M}$, can be regarded as a homogeneous quasi-2D system with height $l_z=(2M+1)H$. Thus one can analogously obtain its corresponding IBC and ELC corrections as was discussed in Section~\ref{subsec::IBCELC}.

First, the energy $U^\text{c}$ can be decomposed via Ewald splitting:
\begin{equation}\label{eq::39}
	U^{\text{c}} \approx  U^{\text{c}}_{M} = U^{\text{c}}_{\text{real}} + U^{\text{c}}_{\text{Fourier}}\;,
\end{equation}
where the real space component $U^{\text{c}}_{\text{real}}$ is given by
\begin{equation}\label{eq::UcReal}
	\begin{split}
		U_{\te{real}}^{\text{c}}=&\frac{1}{2}\sum_{i,j=1}^{N}q_iq_j{\Biggg [}{\bm{n}}{}^\prime \frac{\te{erfc}(\alpha\left|\bm{r}_{ij}+\bm{\ell}\right|)}{\left|\bm{r}_{ij}+\bm{\ell}\right|}+\sum_{\bm{n}}\sum_{l=1}^{M}\frac{\gamma^{(l)}_{+}\te{erfc}\left(\alpha \left|\bm{r}_i-\bm{r}_{j+}^{(l)}+\bm{\ell}\right|\right)}{\left|\bm{r}_i-\bm{r}_{j+}^{(l)}+\bm{\ell}\right|}\\
		&+\sum_{\bm{n}}\sum_{l=1}^{M}\frac{\gamma^{(l)}_{-}\te{erfc}\left(\alpha \left|\bm{r}_i-\bm{r}_{j-}^{(l)}+\bm{\ell}\right|\right)}{\left|\bm{r}_i-\bm{r}_{j-}^{(l)}+\bm{\ell}\right|}\Biggg]\;.
	\end{split}
\end{equation}
In Eq.~\eqref{eq::UcReal}, the additional terms comparing to Eq.~\eqref{eq:ewald2d-1} represent the interaction between the actual charges and the fictitious image charges, the detailed formula for the real space component of force is given by Eq.~(SM1.3) in the Supplementary Materials (SM).
The $\V k$-space component $U^{\text{c}}_{\text{Fourier}}$ is given by
\begin{equation}\label{eq::40}
	\begin{split}
		U^{\text{c}}_{\te{Fourier}}&=\frac{1}{2\alpha L_xL_y}
		\sum_{i,j=1}^{N}q_iq_j\sum_{\V h}{}^\prime e^{\i \V h\cdot \V r_{ij}}\int_{-\infty}^{\infty}\frac{e^{-\frac{h^2}{4\alpha^2}-t^2}}{\frac{h^2}{4\alpha^2}+t^2}\Gamma_{ij}(t)dt-\frac{\alpha}{\sqrt{\pi}}\sum_{i=1}^{N}q_i^2+\mathcal{J}_0\;,
	\end{split}
\end{equation}
where
\begin{equation}
	\Gamma_{ij}(t):=e^{2\i\alpha z_{ij}t}+\sum_{l=1}^{\infty}\left[\gamma^{(l)}_{+}e^{2\i\alpha t(z_i-z_{j+}^{(l)})}+\gamma^{(l)}_{-}e^{2\i\alpha t(z_i-z_{j-}^{(l)})}\right]
\end{equation}
and the $\V 0$-th mode correction $\mathcal{J}_0$ reads
\begin{equation}
	\mathcal{J}_0=\frac{1}{2\alpha L_xL_y}\sum_{i,j=1}^{N}q_iq_j\int_{-\infty}^{\infty}\frac{e^{-t^2}\Gamma_{ij}(t)-\Gamma_{ij}(0)}{t^2}dt.
\end{equation}
Further applying the trapezoidal rule to discretize the integrals, one obtains an efficient summation formula with error estimates, summarized in Theorem~\ref{thm::FourDie}.
\begin{theorem}\label{thm::FourDie} 
	Assume $L_z>H$, discretizing the integrals in Eq.~\eqref{eq::40} via trapezoidal rule with mesh size $\pi/\alpha L_z$, one has
	\begin{equation}\label{eq::UcFour}
		U^{\emph{c}}_{\emph{Fourier}}=\frac{2\pi}{L_xL_yL_z}\sum_{\bm{k}}{}^{\prime}\frac{e^{-\frac{k^2}{4\alpha^2}}}{k^2}\rho_{\bm{k}}\bar{\rho}_{\bm{k}}^{M}-\frac{\alpha}{\sqrt{\pi}}\sum_{i=1}^{N}q_i^2+U_{\emph{IBC}}^{M}+U_{\emph{ELC}}^{M}+U_{\emph{err}}^{M},
	\end{equation}
	where $\bm{k}=2\pi(n_x/L_x,n_y/L_y,n_z/L_z)$, the structure factors $\rho_{\bm{k}}$ and $\bar{\rho}_{\bm{k}}^{M}$ are defined by
	\begin{equation}
		\rho_{\bm{k}}:=\sum_{i=1}^{N}q_ie^{\i \bm{k}\cdot \bm{r}_i}\quad\text{and}\quad\bar{\rho}_{\bm{k}}^{M}:=\sum_{j=1}^{N}q_j\left[e^{-\i \bm{k}\cdot \bm{r}_i}+\sum_{l=1}^{M}\left(\gamma_{+}^{(l)}e^{-\i \bm{k}\cdot \bm{r}_{j+}^{(l)}}+\gamma_{-}^{(l)}e^{-\i \bm{k}\cdot \bm{r}_{j-}^{(l)}}\right)\right],
	\end{equation}
	the second term on the RHS is the self-energy correction, and
	\begin{equation}
		U_{\emph{IBC}}^{M}:=\frac{2\pi}{L_xL_yL_z}\left(\sum_{i=1}^{N}q_{i}z_{i}\right)\sum_{j=1}^{N}q_j\left[z_{j}+\sum_{l=1}^{M}\left(\gamma_{+}^{(l)}z_{j+}^{(l)}+\gamma_{-}^{(l)}z_{j-}^{(l)}\right)\right]\;,
	\end{equation}
	\begin{equation}\label{eq::46}
		U_{\emph{ELC}}^{M}:=\frac{2\pi}{L_xL_y}\sum_{i,j=1}^{N}q_iq_j\sum_{\bm{h}}{}^\prime \frac{e^{\i \bm{h}\cdot\bm{r}_{ij}}}{h}\frac{\cosh(hz_{ij})+\mathcal{G}_{ij}(h)}{1-e^{hL_z}}\;
	\end{equation}
	are the corresponding IBC and ELC correction terms with
	\begin{equation}
		\mathcal{G}_{ij}(h):=\sum\limits_{l=1}^{M}\left[\gamma_{+}^{(l)}\cosh(h(z_i-z_{j+}^{(l)}))+\gamma_{-}^{(l)}\cosh(h(z_i-z_{j-}^{(l)}))\right],
	\end{equation}
	and $U_{\emph{err}}^{M}$ is the remainder error term. 
	The last two terms have decay rates
	\begin{equation}\label{eq::48}
		\left|U_{\emph{ELC}}^{M}\right|\sim \mathcal{O}\left(e^{-\frac{2\pi(L_z-H)}{\max\{L_x,L_y\}}}+\sum_{l=1}^{M}\left[\gamma_{+}^{(l)}+\gamma_{-}^{(l)}\right]e^{-\frac{2\pi(L_z-(l+1)H)}{\max\{L_x,L_y\}}}\right)
	\end{equation}
	and
	\begin{equation}\label{eq::49}
		\left|U_{\emph{err}}^{M}\right|\sim \mathcal{O}\left(e^{-\alpha^2(L_z-H)^2}+\sum_{l=1}^{M}\left(\gamma_{+}^{(l)}+\gamma_{-}^{(l)}\right)e^{-\alpha^2(L_z-(l+1)H)^2}\right)
	\end{equation}
	respectively.
\end{theorem}
\begin{proof}
	Analogous to the proof of Theorem~\ref{thm::err},
	by error analysis for the trapezoidal rule as outlined in Appendix~\ref{app::trapezoidal}, it can be shown that the quadrature discretization of Eq.~\eqref{eq::40} directly yields the first two terms in Eq.~\eqref{eq::UcFour}, as well as the residual correction associated with the ELC term. The estimates for both the ELC term and the remainder error term can be derived directly from Eqs.~\eqref{eq::46} and \eqref{eq::A.8}.
\end{proof}
It should be noted that, unlike the homogeneous case, here the first two terms on the RHS are no longer identical to that of the standard Ewald3D summation, since now the charged system consists of both actually charges and fictitious image charges.


Finally, based on Theorem~\ref{thm::FourDie}, we obtain the following criterion for selecting appropriate values of $M$ and $L_z$ to ensure that the contributions from the image truncation error, the ELC term, and the remainder error term are all below a specified tolerance $\varepsilon$:
\begin{itemize}
	\item First, choose the image reflection truncation level $M$ according to Eq.~\eqref{eq::38}, such that $|U^{\text{c}}-U^{\text{c}}_{M}|\leq \varepsilon$;
	\item Second, choose the extended height of the domain $L_z$  according to  Eqs.~\eqref{eq::48}-\eqref{eq::49}, such that both $\left|U_{\text{ELC}}^{M}\right|\leq \varepsilon$ and $\left|U_{\text{err}}^{M}\right|\leq \varepsilon$ are satisfied.
\end{itemize}
In the worst scenario where $|\gamma_{\text{top}}|=|\gamma_{\text{bot}}|=1$, the decay rates of $|U_{\text{ELC}}^{M}|$ and $|U_{\text{err}}^{M}|$ are solely determined by the exponential factors, as evident from Eqs.~\eqref{eq::48} and \eqref{eq::49}. It is found that $L_z$ can be chosen according to
\begin{equation}\label{eq::LzM}
	L_z\geq (M+1)H+\max\left\{\frac{\max\{L_x,L_y\}}{2\pi}\log\frac{1}{\varepsilon},~\frac{1}{\alpha}\sqrt{\log\frac{1}{\varepsilon}}\right\}\;.
\end{equation}
Particularly, it is observed that $L_z$ needs to be larger than $(M+1)H$, 
this can be understood as a physical constraint which prevents the $z$-replicas of actual charges overlapping with their image charges. 
Eq.~\eqref{eq::LzM} also explains the counter-intuitive convergence result in~\cite{yuan2021particle}, where it was observed that increasing $M$ (while keeping $L_z$ fixed) may lead to a non-monotonic behavior in accuracy, i.e., the error first decays as $M$ increases; but eventually grows exponentially when $M$ is large. 
This can now be understood that if one keeps increase $M$ but with $L_z$ fixed, then the constraint $L_z\geq (M+1)H$ will eventually be violated, then the approximation is no longer stable.

\subsection{RBE2D method for dielectrically confined quasi-2D systems} \label{subsec::IBCELCDielectric}

We now extend the RBE2D method for efficient MD simulations of dielectrically confined quasi-2D Coulomb systems. 
First, given tolerance $\varepsilon$, one shall choose appropriate values for $M$ and $L_z$, such that the image series truncation error, the ELC term, and remainder error term can all be neglected. Then, let $\bm{F}_{\text{Fourier}}^{\text{c}}$ denotes the $\V k$-space component of force acting on the $i$-th particle, it reads
\begin{equation}\label{eq::52}
	\bm{F}_{\text{Fourier}}^{\text{c}}(\bm{r}_i)=-\frac{2\pi}{L_xL_yL_z}\sum_{\bm{k}}{}^{\prime}\frac{e^{-\frac{k^2}{4\alpha^2}}}{k^2}\nabla_{\bm{r}_i}\left[\rho_{\bm{k}}\bar{\rho}_{\bm{k}}^{M}\right]+\bm{F}_{\text{IBC}}^{M}(\bm{r}_i)\;,
\end{equation}
where 
\begin{equation}\label{eq::rhorhoM}
	\begin{split}
		\nabla_{\bm{r}_i}\left[\rho_{\bm{k}}\bar{\rho}_{\bm{k}}^{M}\right]=&q_i\bm{k}\text{Im}\left[e^{-\i\bm{k}\cdot\bm{r}_i}\left(\rho_{\bm{k}}+\rho_{\bm{k}}^{M}\right)+\sum_{\substack{l=1\\ \text{even}}}^{M}\left(\gamma_{+}^{(l)}e^{-\i\bm{k}\cdot\bm{r}_{i+}^{(l)}}+\gamma_{-}^{(l)}e^{-\i\bm{k}\cdot\bm{r}_{i-}^{(l)}}\right)\rho_{\bm{k}}\right]\\
		+&q_i\widehat{\bm{k}}\text{Im}\left[\sum_{\substack{l=1\\ \text{odd}}}^{M}\left(\gamma_{+}^{(l)}e^{-\i\bm{k}\cdot\bm{r}_{i+}^{(l)}}+\gamma_{-}^{(l)}e^{-\i\bm{k}\cdot\bm{r}_{i-}^{(l)}}\right)\rho_{\bm{k}}\right]
	\end{split}
\end{equation}
with $\rho_{\bm{k}}^{M}$ the conjugate of $\bar{\rho}_{\bm{k}}^{M}$ and $\widehat{\bm{k}}=(k_x,k_y,-k_z)$. The calculation of Eq.~\eqref{eq::52} can also be accelerated via FFT, namely, the ICM-PPPM method~\cite{yuan2021particle}. 
If the image reflection level is truncated at $M$, the computational cost of ICM-PPPM is $\mathcal O(MN \log(MN))$. Thus for strongly confined systems, the cost of FFT-based method will increase rapidly as $M$ increases.

We now again apply the importance sampling technique to approximate $\bm{F}_{\text{Fourier}}^{\text{c}}$, yielding the following stochastic estimator as 
\begin{equation}\label{eq::important}
	\bm{F}_{\text{Fourier}}^{\text{c}}(\bm{r}_i)\approx \bm{F}_{\text{Fourier}}^{\text{c},*}(\bm{r}_i)=-\frac{2\pi }{L_xL_yL_z}\sum_{\ell=1}^{P}\frac{S}{P}\frac{\nabla_{\bm{r}_i}\left[\rho_{\bm{k}_{\ell}}\bar{\rho}_{\bm{k}_{\ell}}^{M}\right]}{k_{\ell}^2}+\bm{F}_{\text{IBC}}^{M}(\bm{r}_i),
\end{equation}
where $\{\bm{k}_{\ell}\}_{\ell=1}^{P}$ are the $P$ batches of wave vectors sampled from $\mathscr{P}(\bm{k})$. 
Note that if surface charges are present on the interfaces, a simple modification for Eq.~\eqref{eq::important} is needed, detailed information can be found in the SM, Section S1. And the real space component of force is given by Eq.~(
SM1.3), which can still be evaluated by neighbor lists algorithms.


Next, we describe a simple strategy for efficiently summing over the $2MN$ image charges.
Notice that Eq.~\eqref{eq::rhorhoM} can be rewritten as
\begin{equation}\label{eq::56}
	\begin{split}
		\nabla_{\bm{r}_i}\left[\rho_{\bm{k}}\bar{\rho}_{\bm{k}}^{M}\right]
		= & q_i\bm{k}\text{Im}\left[e^{-\i \bm{k}\cdot \bm{r}_i}\sum_{j=1}^{N}q_je^{\i\bm{k}\cdot\bm{r}_{j}}\left(1+e^{-2\i k_zz_j}Y_{\text{odd}}(k_z)+Y_{\text{even}}(k_z)\right)\right]\\
		+ & q_i\bm{k}\text{Im}\left[e^{-\i \bm{k}\cdot\bm{r}_i}\left(1+e^{2\i k_zz_i}\overline{Y}_{\text{odd}}(k_z)\right)\rho_{\bm{k}}\right]+q_{i}\widehat{\bm{k}}\text{Im}\left[e^{-\i\bm{k}\cdot\bm{r}_i}\overline{Y}_{\text{even}}(k_z)\rho_{\bm{k}}\right],
	\end{split}
\end{equation}
where the coefficients $Y_{\text{odd}}$ and $Y_{\text{even}}$ are defined by $\sum_{\substack{l=1, \text{odd}}}^{M}\left[\gamma_{+}^{(l)}e^{\i k_z(l+1) H}+\gamma_{-}^{(l)}e^{-\i k_z (l-1)H}\right]$ and $\sum_{\substack{l=1, \text{even}}}^{M}\left[\gamma_{+}^{(l)}e^{\i k_zl H}+\gamma_{-}^{(l)}e^{-\i k_z l H}\right]$, respectively. Clearly, now the summation over~$j$ and~$l$ are decoupled.
In a simulation, one shall precompute $Y_{\text{odd}}(k_z)$ and $Y_{\text{even}}(k_z)$ for each $\bm{k}\in \{\bm{k}_{\ell}\}_{\ell=1}^{P}$, and subsequently given a particle configuration, evaluate $\nabla_{\bm{r}_i}\left[\rho_{\bm{k}}\bar{\rho}_{\bm{k}}^{M}\right]$ for all $i$ based on Eq.~\eqref{eq::56}. 
As such, the total computational cost in force evalulations is reduced to $\mathcal{O}(P(M+N))$.
In practice, one shall still take $\alpha \sim \rho_{r}^{1/3}$ to ensure that the real space calculation costs $\mathcal{O}(N)$, then the $\V k$-space components can be calculated according to Eq.~\eqref{eq::56} so that the complexity is reduced to $\mathcal{O}(P(M+N))$.
Thus, the overall complexity of the RBE2D is approximately $\mathcal{O}(N)$ since both $P$ and $M$ are of $\mathcal O(1)$ and independent with $N$.
The detailed algorithm for RBE2D accelerated MD in the presence of dielectric interfaces is summarized in Algorithm~\ref{alg::RBEalg}.

\begin{algorithm}[H]
	\caption{(RBE2D for quasi-2D systems under dielectric confinement)}\label{alg::RBEalg}
	\begin{algorithmic}[1]
		\State Choose $\alpha$, $r_c$ (the cutoff in real space), $\Delta t$, batch size $P$, image reflection level of truncation $M$, and the vacuum layer parameter $L_z$. Initialize the coefficients of dielectric jumps ($\varepsilon_{\rm{top}}$, $\varepsilon_{\rm{c}}$, $\varepsilon_{\rm{bot}}$) and the positions and velocities of charges $\bm{r}^0_i, \bm{v}^0_i$ for $1\le i\le N$. Set $N_T$ as the total MD time steps.
		\For {$n \text{ in } 1: N_T$}
		\State Sample $P$ nonzero frequencies $\bm{k}\sim e^{-k^2/(4\alpha)}$ by the Metropolis procedure to form set $\mathcal{K}$.
		\State The real space component of force is computed via a neighbor list algorithm (for both actual charges and image charges within $r_c$).
		\State The Fourier space component of force is computed via the random batch approximation $\bm{F}_{\text{Fourier}}^{\text{c},*}$, with $P$ frequencies sampled from $\mathcal{K}$.
		\State Integrate the equations of motions for time $\Delta t$ with appropriate integration scheme and some appropriate thermostat. 
		\EndFor
	\end{algorithmic}
\end{algorithm}

Finally, consider the fluctuation of random batch approximation of the force, denoted as $\bm{\Xi}_{\bm{F},i}^{\text{c}}=\bm{F}_{\text{Fourier}}^{\text{c},*}(\bm{r}_i)-\bm{F}_{\text{Fourier}}^{\text{c}}(\bm{r}_i)$, one can still show that 1) the variance of the RBE2D method is controlled via $\mathcal{O}(1/P)$ and is independent of $L_z$; 2) the RBE2D is capable to capture finite time structure and dynamic properties in MD simulations. 
The proof is similar to Theorems~\ref{Thm::1} and~\ref{thm::2}, and is thus omitted here.



\section{Numerical results on quasi-2D systems without dielectric interfaces}\label{sec::numerical}

In this section, numerical results are first presented for the homogeneous case, validating the performance of RBE2D accelerated MD simulations. 
Examples include coarse-grained model of electrolytes and all-atom SPC/E bulk water systems~\cite{berendsen1987missing} confined by two slabs without dielectric jumps. 
Our method is implemented in LAMMPS (version 29Oct2020) \cite{thompson2021lammps} and performed on the ``Siyuan Mark-I'' high performance cluster at Shanghai Jiao Tong University. 
Each node is equipped with $2\times$ Intel Xeon ICX Platinum $8358$ CPU ($2.6$GHz, $32$ cores) and $512$ GB memory.
The code is highly optimized, where message passing interface (MPI) and Intel AVX-512 instructions are used for parallelization and vectorization, respectively.

To benchmark the accuracy, several widely investigated quantities are computed, such as concentration, mean-squared displacement (MSD), velocity auto-correlation functions (VACF), ensemble-averaged energy distribution, etc. 
The calculation methods for these quantities are described in Section~SM2 of the SM. 
For comparison, we used the stable implementation of the PPPM2D method~\cite{crozier2001molecular} provided inside LAMMPS. 
Results obtained using the PPPM2D method are labeled as ``PPPM''. 
The parameters of the PPPM were automatically chosen in LAMMPS to achieve a desired relative error of $\Delta=10^{-4}$ for force evaluations~\cite{deserno1998mesh}. 
As a fair comparison, the Ewald splitting parameter $\alpha$ is always set to be the same for RBE2D and PPPM methods.

\subsection{Accuracy test case I: 3:1 electrolytes immersed in an implicit solvent}\label{subsec::electrolyte-neutral}
We first perform coarse-grained MD simulations of $3:1$ electrolytes in the NVT ensemble. Following the primitive model, all the ions are represented as soft spheres of diameter $\sigma$ and mass $m$ that interact via a purely repulsive shifted-truncated Lennard-Jones (LJ) potential
\begin{equation}
	U_{\text{LJ}}=\begin{cases}
		4\varepsilon_{\text{LJ}}\left[\left(\dfrac{\sigma}{r_{ij}}\right)^{12}-\left(\dfrac{\sigma}{r_{ij}}\right)^6+\dfrac{1}{4}\right],~~~~r_{ij}\leq r_{\text{LJ}},\\
		0,~~~~r_{ij}>r_{\text{LJ}},	
	\end{cases}
\end{equation}
where $r_{ij}$ is the distance between $i$-th and $j$-th particles, $r_{\text{LJ}}=2^{1/6}\sigma$ the LJ truncation distance, $\varepsilon_{\text{LJ}}=k_{\text{B}}T$ the coupling strength, $k_{\text{B}}$ is the Boltzmann constant, and $T$ the temperature. These ions are immersed in a continuum solvent, characterized by a Bjerrum length 
$\ell_{\text{B}}=e^2/(4\pi \varepsilon_c k_{\text{B}}T)=3.5\sigma$. The system has dimensions $L_x=90\sigma$, $L_y=90\sigma$, and $H=30\sigma$, containing $750$ cations and $2250$ anions. The ions are confined in $z$ by purely repulsive LJ walls at $z=0$ and $z=30\sigma$, with parameters $\varepsilon_{\text{wall}}=\varepsilon_{\text{LJ}}$ and $\sigma_{i,\text{wall}}=0.5\sigma$. 
The initial $10^6$ steps with $\Delta t=10^{-3}\tau$ are performed for equilibrium, where $\tau=\sqrt{m\sigma^2/\varepsilon_{\text{LJ}}}$ denotes the LJ unit of time. The production phase follows for another $10^7$ steps, and the configurations are sampled every $100$ time steps for statistics. 
The temperature is maintained by using a Nos\'e-Hoover thermostat~\cite{hoover1985canonical} with relaxation times $0.01\tau$, 
and $r_c$ is set as $10\sigma$ for both the RBE2D and PPPM methods.  

\begin{figure}[!ht]
	\centering
	\includegraphics[width=0.90\linewidth]{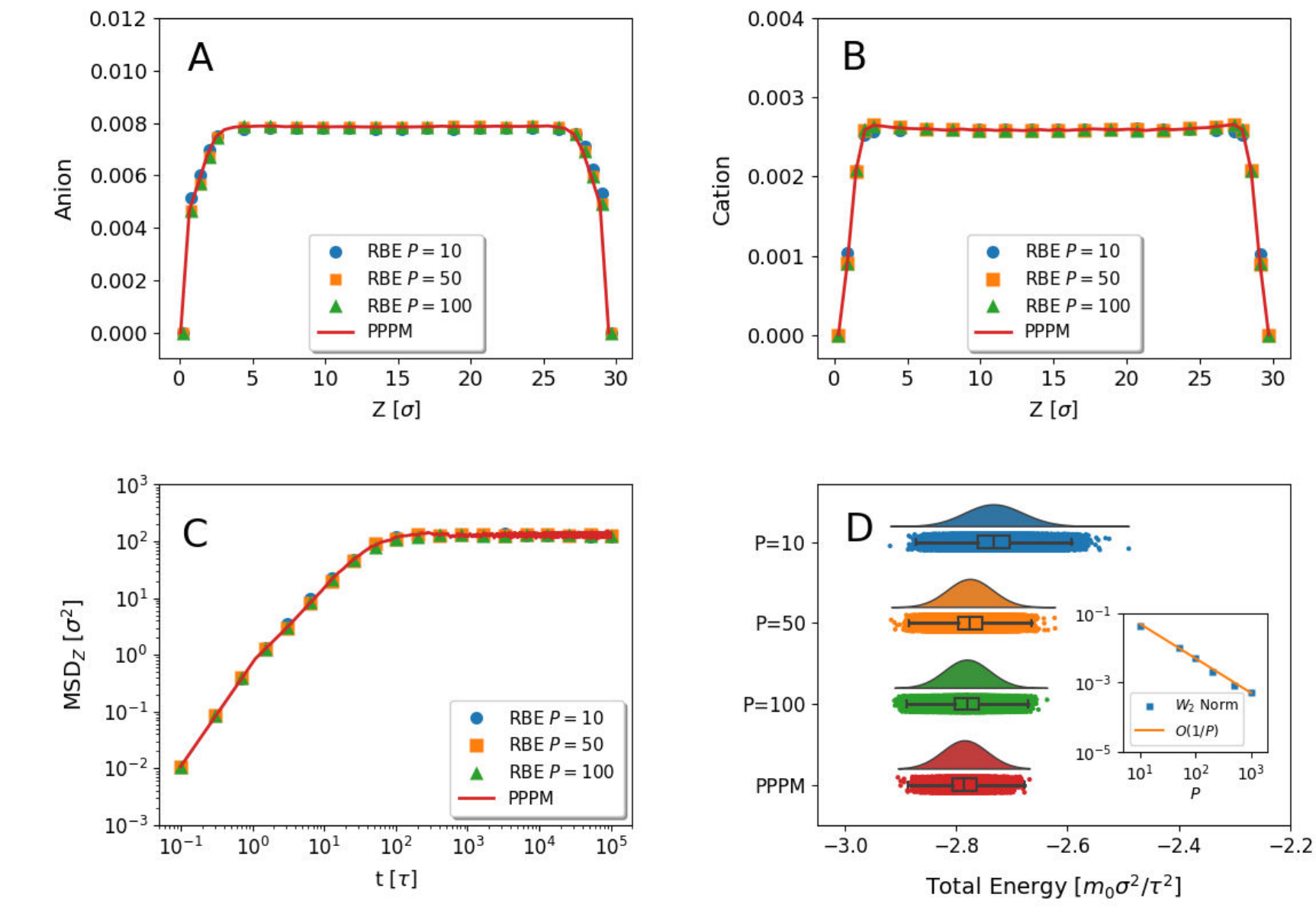}
	\caption{MD results for $3:1$ electrolytes. (a-b): The variance of force and the error on ensemble average of the potential energy as a function of the relative size of the extended system, $L_z/H$. (c): The concentration of cations along $z$-dimension. (d): The raincloud plots of the ensemble average distribution of the potential energy as well as a boxplot and raw data-points. In (a), (c) and (d), results are shown for the RBE2D with different batch size $P$ and the PPPM. In (b), we use the $W_2$ norm to measure the difference between two distributions. The inset in (d) shows the convergence on the $W_2$ norm with $\mathcal{O}(1/P)$ rate.}
	\label{fig:3_1}
\end{figure}

First, let us consider the ensemble average of the variance in forces, denoted as $\langle\|\mathbb{E}|\bm{\Xi}_i|^2\|_{\infty}\rangle$,
results are shown in Fig.~\ref{fig:3_1}(a). 
Clearly, the force variance is independent with the vacuum layer thickness parameter $L_z$, which validates our theoretical result stated in Theorem~\ref{Thm::1}. 
The distribution of potential energy $\mathscr{P}(U)$, is also evaluated.
We compared the sampled distribution by both the RBE2D and PPPM, denoted as~$\mathscr{P}_{r}$ and $\mathscr{P}_{p}$, respectively.
To measure the difference between two probability distributions, we consider the Wasserstein-2 norm~\cite{santambrogio2015optimal}, defined as:
\begin{equation}
	W_2(\mathscr{P}_{r}, \mathscr{P}_{p})=\left(\inf _{\gamma \in \Pi(\mathscr{P}_{r}, \mathscr{P}_{p})} \int_{\mathbb{R} \times \mathbb{R}}|x - y|^2 d \gamma\right)^{1 / 2}\;,
\end{equation}
where $\Pi(\mathscr{P}_{r}, \mathscr{P}_{p})$ represents the set of all joint distributions with marginal distributions $\mathscr{P}_{r}$ and $\mathscr{P}_{p}$, respectively. 
The corresponding results are shown Fig.~\ref{fig:3_1}(b), where the~$W_2$ norms are poltted as functions of~$L_z / H$ with varies batch size~$P$.
According to Theorem~\ref{thm::err}, when $L_z\approx H$, the errors coming from the ELC and remainder error term are still large and can not be ignored; which is consistent with the error decay observed here. 
We also find that, as $L_z\geq 3H$, the error no longer decreases as $L_z/H$ increases (with $P$ fixed); indicating that
both ELC and remainder error terms now indeed become negligible, and the error primarily comes from the random batch approximation. 

Finally, the concentration of trivalent ions along $z$ and the distribution of energy are investigated.
The corresponding results are depicted in Fig.~\ref{fig:3_1}(c-d). 
We observe a convergence rate of $\mathcal{O}(1/P)$ in the average energy by the RBE2D (with that of the PPPM set as benchmark value), 
which is consistent with our theoretical prediction Theorem~\ref{Thm::1}.
It is also noticed that, for the RBE2D method, choosing $P=50$ can already yield very accurate results, almost indistinguishable from those obtained via PPPM. 

\subsection{Accuracy test case II: all-atom SPC/E bulk water system}
\label{sec::waterhomo}

Next, we perform MD simulations for the SPC/E bulk water system~\cite{berendsen1987missing}. 
The system has dimensions $L_x=L_y=H=81.9~\mathring{\text{A}}$, and consists of $53367$ atoms. 
Two purely repulsive shifted-truncated LJ walls (with $\varepsilon_{\text{atom-wall}}= k_{\text{B}}T)$ are located at $z=0$ and $z=H$. 
The simulations utilize a time step $\Delta t=1fs$, and $T$ is fixed at $298K$ by a NH thermostat with relaxation time $10fs$. 
Equilibration proceeds for $10^5$ time steps, followed by a $10^6$ steps production period, with sampling made every $100$ time steps. 
The cutoff is set as $r_c=10\mathring{A}$, and the ratio $L_z/H$ is fixed as $3$ to ensure that both the ELC and remainder error term are negligible.

\begin{figure}[ht!]
	\centering
	\includegraphics[width=0.90\linewidth]{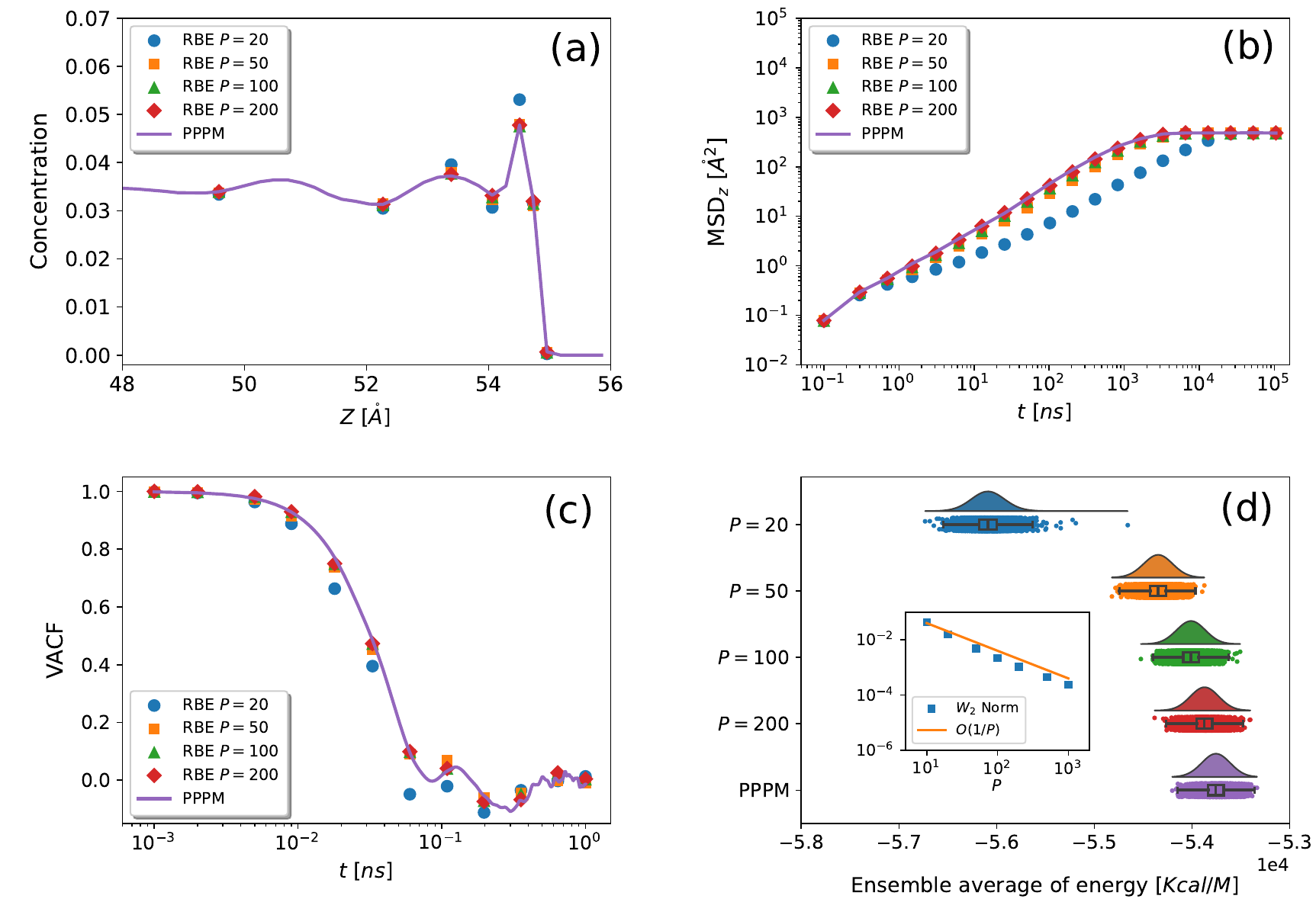}
	\caption{MD results for a SPC/E bulk water system. (a) The concentration of oxygen near the top surface; (b) the MSD along $z$-direction; (c) the VACF; and (d) the raincloud plots of the ensemble average distribution of the potential energy as well as a boxplot and raw data-points. The RBE2D with different batch size $P$ and the PPPM are plotted. The inset in (d) shows the convergence on the $W_2$ norm with $\mathcal{O}(1/P)$ rate.}
	\label{fig:Water}
\end{figure}

We analyze various equilibrium and dynamic properties of the system, including the oxygen concentration, the MSD in $z$, the VACF, and the distribution of potential energy $U$.
The results are summarized in Fig.~\ref{fig:Water}(a-d).
Remarkably, both the RBE2D and PPPM methods produce highly consistent results for all these properties when $P\geq 100$,
demonstrating the capability of RBE2D in capturing the spatiotemporal information across multiple scales in MD simulations.

\subsection{CPU time performance}
We now report the CPU time performance of the RBE2D method, compared to the PPPM method. 
The SPC/E bulk water systems are simulated, where
the system dimensions are set as $L_x=L_y=H$,  and the system size changes as one varies $N$ with the density of water being fixed at $1g/cm^3$. For a fair comparison, a relative error threshold of $10^{-4}$ is set for both methods. 

\begin{figure}[ht!]
	\centering	\includegraphics[width=0.9\linewidth]{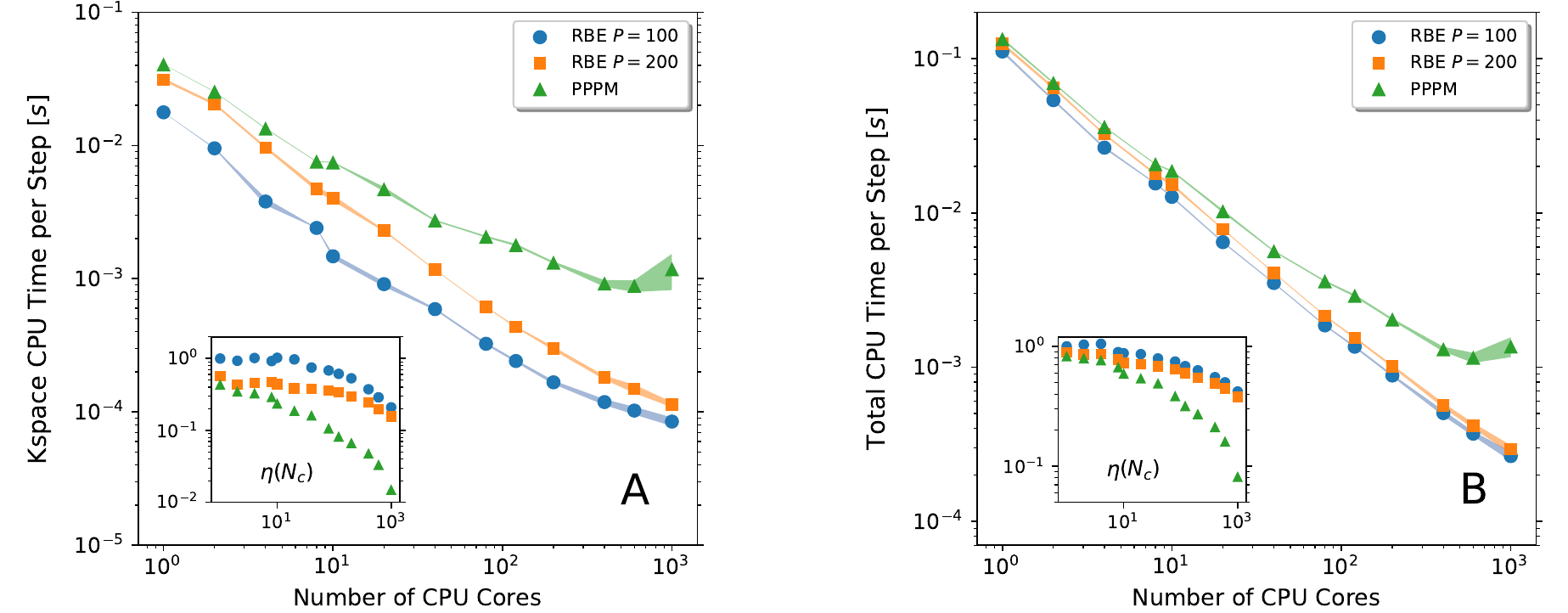}
	\caption{
		(a) CPU time per step for the RBE2D method with increasing $N$ and (b) total CPU time for comparison between the RBE2D and the PPPM with the number of CPU cores $N_{\text{cpu}}$ up to $1024$. (a) is the results with $64$ cores and $P=100$, where the solid lines indicate linear-scaling. In (b), data are shown for the RBE2D with batch size $P=100$ and $P=200$ and the PPPM, where the light-colored parts stand for the range of error bar. The inset in (b) shows the strong scalability $\eta(N_{\text{cpu}})$.}
	\label{fig:Time}
\end{figure}

We first investigate the time cost and complexity by varying particle number $N$. 
In Fig.~\ref{fig:Time}(a), 
the  CPU time per MD step for the RBE2D is documented with particle number $N$ up to $2\times 10^6$, where the solid lines indicate linear fitting of the data.
The results clearly validate the $\mathcal{O}(N)$ complexity of the RBE2D method. 
For FFT-based methods, the CPU time of the real space and the Fourier space computations are balanced to be roughly the same. 
Here we clearly observe that RBE2D has significantly reduced the CPU time cost in the Fourier part over the whole range of $N$. 

Next, we directly compare the RBE2D and PPPM methods, in terms of both
CPU time performance and scalability. 
Let $N_{\text{cpu}}$ be the utilized CPU number, and $T(N_{\text{cpu}})$ the corresponding CPU time,
the so-called ``strong scaling'' is defined as
\begin{equation}\label{eq::etau}
	\eta(N_{\text{cpu}})=\frac{N_{\text{min}}}{N_{\text{cpu}}}\frac{T_{\text{min}}}{T(N_{\text{cpu}})},
\end{equation}
where $T_{\text{min}}=T(N_{\text{min}})$ and $N_{\text{min}}$ is the minimal number of CPU used. 
Fig.~\ref{fig:Time}(b) and its inset document the CPU time and strong scaling results by the RBE2D and PPPM methods, for a system comprising $139968$ atoms. Clearly,
the two methods perform similarly when $N_{\text{cpu}}$ is small; and RBE2D outperforms PPPM as $N_{\text{cpu}}$ increases, the advantage becomes significant for $N_{\text{cpu}}>64$. Notably,
when $N_{\text{cpu}}\geq 1024$, RBE2D is approximately 10 times faster than PPPM. 
Furthermore, the strong scaling of the RBE2D remains over $70\%$ when~$\sim 4000$ CPUs are employed, while that of the PPPM drops raplidly to $6\%$. 
This highlights the excellent parallel scalability of the RBE2D method. 

\section{Numerical results on quasi-2D systems under dielectric confinement}\label{sec::numericalDielectric}

In this section, numerical simulations are performed for systems in the presence of dielectric interfaces, including dielectrically confined electrolytes and SPC/E bulk water, to further validate the RBE2D method. 
As a benchmark for accuracy and efficiency, the simulations are also  conducted using the HSMA method~\cite{liang2020harmonic} and the ICM-PPPM method~\cite{yuan2021particle}, both have been implemented in LAMMPS~\cite{liang2022hsma}. 
For the RBE2D, parameters such as $M$ and $L_z$ are selected according to Section~\ref{subsec::IBC&ELC&ICM}; while
for the HSMA and ICM-PPPM, we use the same parameters suggested by Refs.~\cite{liang2020harmonic} and~\cite{yuan2021particle}, respectively. For all methods, the threshold in  relative errors is set to $10^{-4}$.

\subsection{Accuracy test case III: dielectrically confined electrolytes}

Here we examine four distinct electrolyte systems that are confined by two dielectric interfaces.
These systems include 2:1 and 3:1 electrolytes with varying surface charge densities and dielectric contrasts.
In these systems, all ions are modeled as soft spheres of diameter $r_{\text{d}}$, and are confined by purely repulsive shifted-truncated LJ walls ($\varepsilon_{\text{ion-wall}}=k_{\text{B}}T$; $\sigma_{\text{ion-wall}}=0.5r_{\text{d}}$) at $z=0$ and $z=H$. More detailed information regarding the system settings can be found in Table~SM1 in the SM. 
The MD simulations utilize a time step $0.005\tau$, with temperature controlled by a NH thermostat 
with a damping time of $0.05\tau$. 
All simulations start with an equilibration period of $10^6$ time steps, followed by a production period of $10^8$ time steps.

\begin{figure}[ht!]
	\centering
	\includegraphics[width=0.95\linewidth]{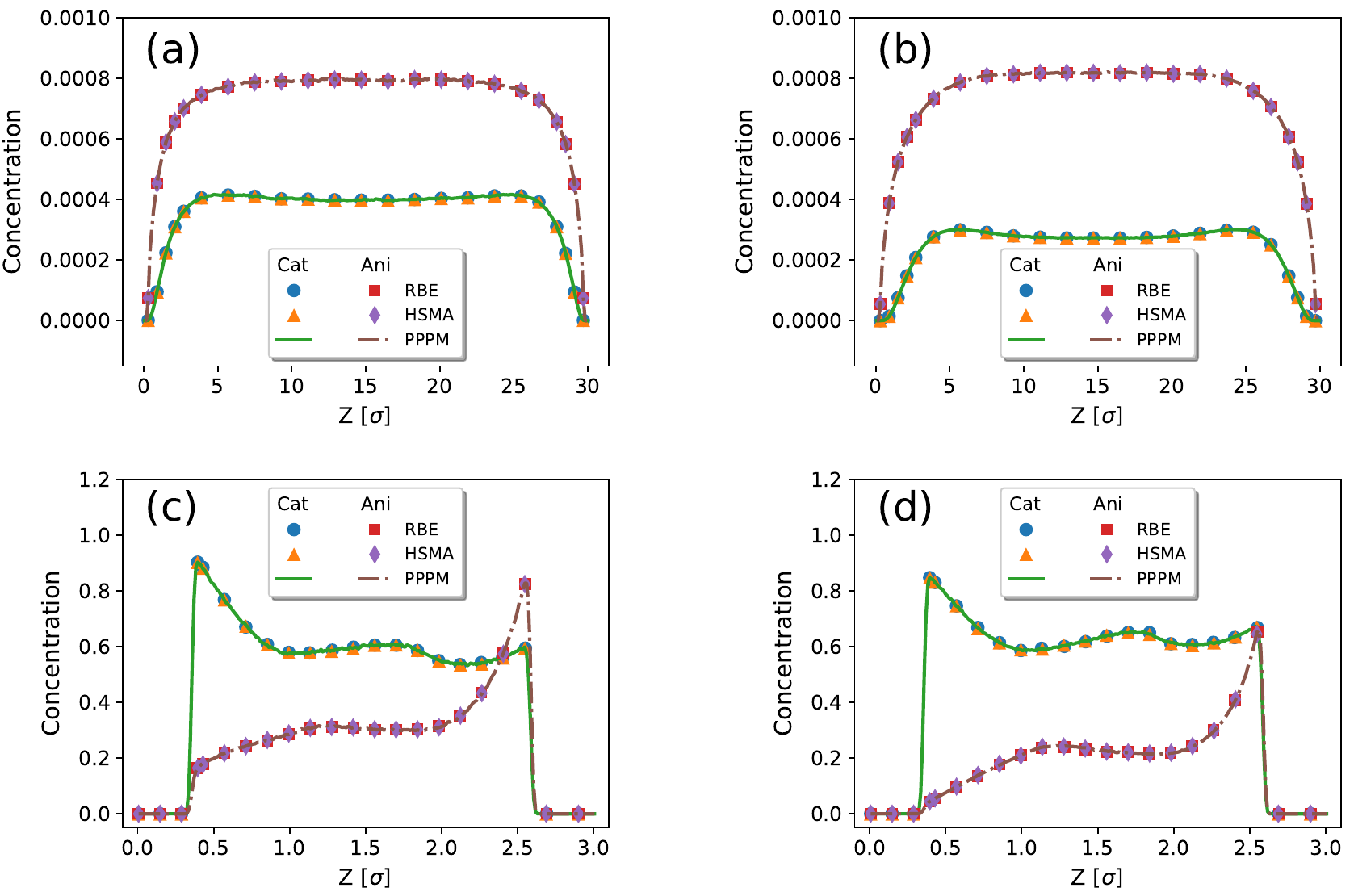}
	\caption{Ionic distributions for [(a) $2:1$ electrolyte and (b) $3:1$ electrolyte] between neutral interfaces with symmetric dielectric contrast, and [(c) $2:1$ electrolyte and (d) $3:1$ electrolyte] confined between walls with asymmetric dielectric contrasts and non-neutral asymmetric surface charge densities. More details about the system settings are provided in Tab.~SM1 in the SM. Data are shown for the RBE2D method with batch size $P=100$, and the HSMA and the ICM-PPPM methods with relative error threshold $10^{-4}$.} 
	\label{fig:den1}
\end{figure}

We first calculate the cation and anion distributions along $z$-axis for all considered systems, the results are summarized in Figure~\ref{fig:den1}(a-d). 
In panels (a-b), where the dielectric contrasts are set as $\gamma_{\text{top}}=\gamma_{\text{bot}}=0.939$, the concentrations illustrate the so-called ``ion depletion effect'' near the insulator-like interfaces, consistent with previous findings. 
Panels (c-d) demonstrate the complicated interplay between the dielectric confinement effect, interfacial charges, and ion valences, and their accumulative effects on the cation distribution. 
For all cases considered, the RBE2D results show excellent agreement with that of the HSMA and ICM-PPPM methods. 
The ion distribution shown in panel (c) also agrees with the result from a boundary element solver~\cite{wu2018asymmetric}, and reported in \cite{yuan2021particle}, figure 4(b). 
Additionally, we calculate the potential energy distribution $\mathscr{P}(U)$ for these systems and measure the discrepancy using the $W_2$ norm.
A convergence rate of $\mathcal{O}(1/P)$ is again observed (see Fig.~\ref{fig:energy2d} in the SM).

\subsection{Accuracy test case IV: dielectrically confined SPC/E Water}

We conclude the numerical tests with simulations of dielectrically confined SPC/E bulk water, and cross compare the RBE2D and ICM-PPPM methods. 
One can refer to Sec.~\ref{sec::waterhomo} for most of the system setups.
Here the dielectric mismatches are introduced, set as $\gamma_{\text{top}}=\gamma_{\text{bot}}=-0.5$. 
The parameters of the ICM-PPPM method were automatically chosen~\cite{kolafa1992cutoff} to achieve a relative error of $\Delta=10^{-4}$, and the real-space cutoff distance $r_c$ was set to $12\mathring{A}$, which was selected to optimize the performance of the ICM-PPPM method. 

\begin{figure}[ht]
	\centering
	\includegraphics[width=0.95\linewidth]{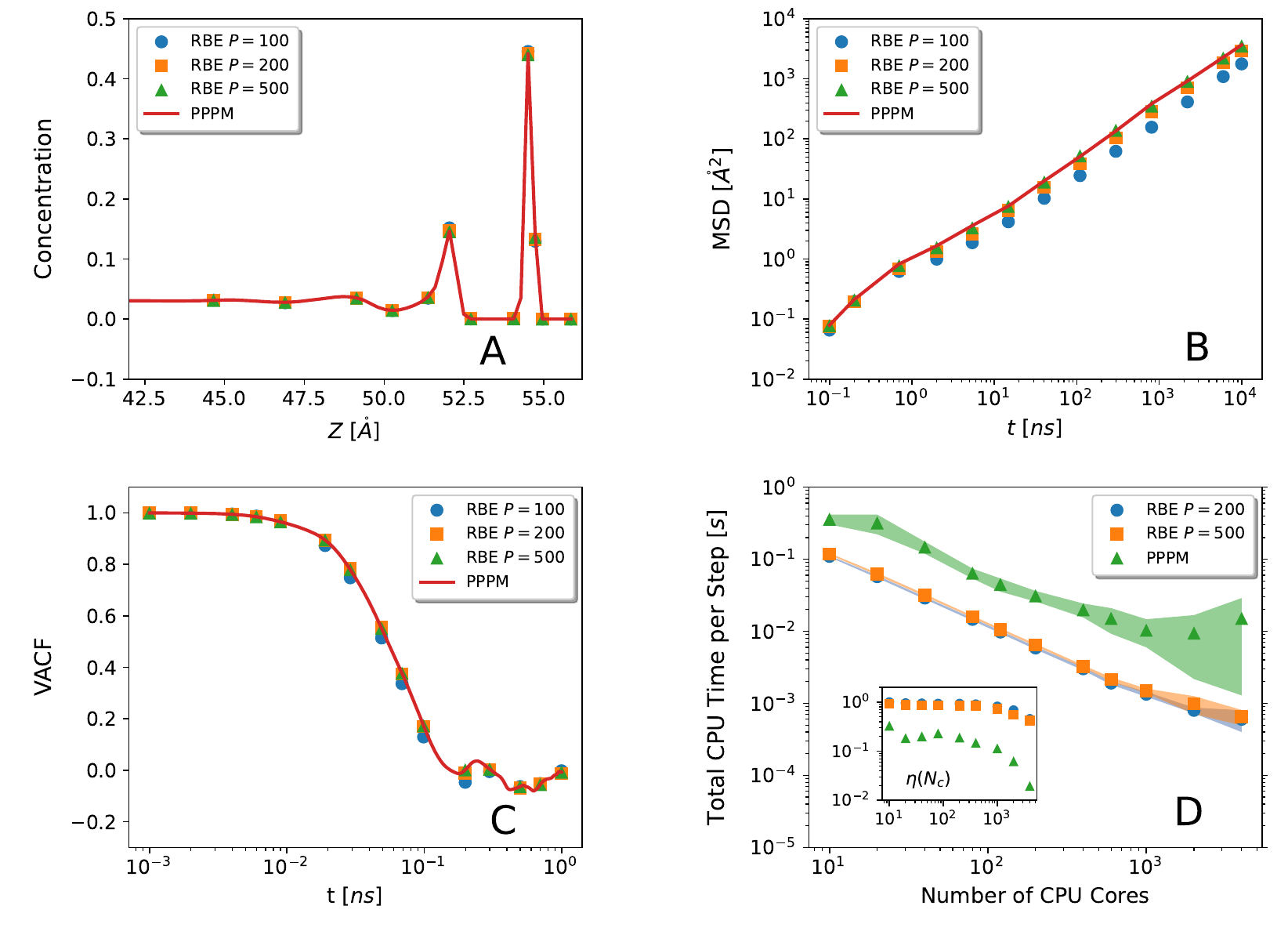}
	\caption{Distribution of oxygen atom in SPC/E bulk water system (a), the MSD (b), the VACF (c), and comparison of time performance between RBE2D (with different $P$) and ICM-PPPM (d) are studied. The inset in (d) shows corresponding strong parallel efficiency $\eta(N_{\text{cpu}})$ defined as Eq.~\eqref{eq::etau}. The results derived from the RBE2D is almost identical to those from the ICM-PPPM. The RBE2D-based MD has about $1.5$ magnitude faster than the ICM-PPPM-based MD.}
	\label{fig:den2}
\end{figure}

We first examine the concentration of oxygen along $z$-axis for a system comprising $53367$ atoms, the results are displayed in Fig.~\ref{fig:den2}(a). 
We again find an excellent agreement between the RBE2D and the ICM-PPPM when $P\geq 200$, whereas some small discrepancy is observed if one chooses $P=100$. 
The MSD and VACF are also calculated and shown in Figs.~\ref{fig:den2}(b-c).
The results clearly indicate that for dynamic properties of the water molecules (such as MSD and VACF), the results obtained by RBE2D is consistent to that of the ICM-PPPM cross multiple time scales ($fs$ to $ns$). 
Finally, the CPU time and scalability data of both the RBE2D and ICM-PPPM methods are documented in Fig.~\ref{fig:den2}(d), based on simulating a water system containing $139968$ atoms. 
It is observed that, the RBE2D is more than 10 times faster than the ICM-PPPM method. 
Furthermore,
the strong scaling is shown in the inset of Fig.~\ref{fig:den2}(d), which indicates 
that RBE2D can remain a strong scaling of $60\%$ for up to $4000$ CPU cores, while that of ICM-PPPM drops to only about $5\%$, again demonstrating the strong scalability of RBE2D method.

\section{Conclusion}\label{sec::conclusion}
In summary, we have developed a novel RBE2D method to efficiently simulate a collection of charges in a quasi-2D geometry, with the extension to scnarios with dielectric confinement.
Our method adopts an importance sampling strategy in $\V k$-space and scales optimally as $\mathcal O(N)$ and has reduced variations.
The accuracy, efficiency and scalability of our method are demonstrated via various numerical tests. 
Comparing with existing method, an excellent agreement in the simulation results is observed, but with a significant reduction in the computational cost and strong scalability thanks to the stochastic sampling nature of our method. 
Thus, our method can enable highly efficient and accurate simulations for large-scale simulations of charged systems under quasi-2D confinement.

When the confined interfaces are not of planar shape, we anticipate that a boundary
integral formulation can be incorporated -- the polarization effect can be approximated as the field due to induced surface charges, then in each field evaluations, RBE2D may be applied to reduce the computational cost, so that it will be capable of efficiently and accurately simulate charged particles confined by structured surfaces~\cite{wu2018asymmetric}.

In addition, we plan to apply the RBE2D method to more complex systems, such as the membrane-protein systems and the battery-electrolyte systems, where the confinement effects can play a crucial role in the system dynamics.


\appendix

\begin{figure}[!ht]
	\begin{center}
		\includegraphics[width=0.8\textwidth]{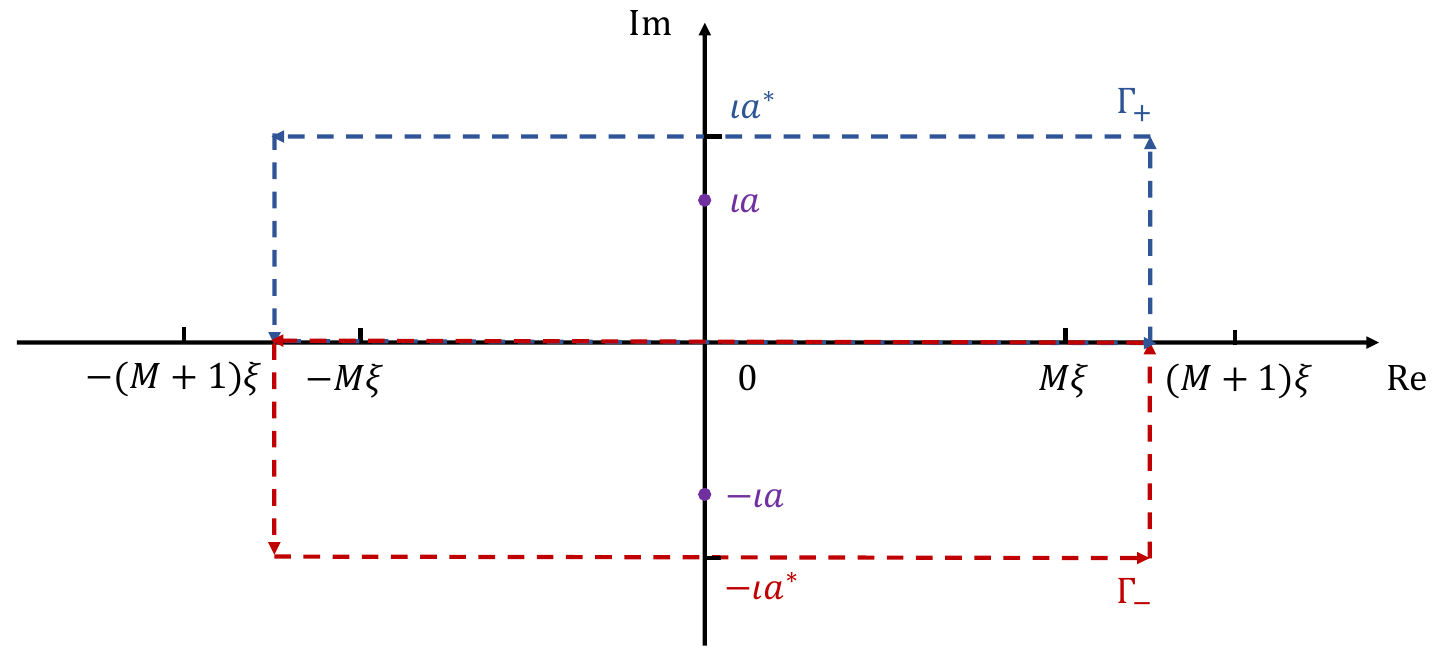}
		\caption{Integration contours in the error estimation of trapezoidal rule.}
		\label{fig:Trapezoidal}
	\end{center} 
\end{figure}

\section{Quadrature error of the trapezoidal rule}\label{app::trapezoidal}
We employ the method of contour integrals~\cite{Donaldson1972SINUA,trefethen2014Rev} to derive a precise estimation for the error in trapezoidal rule when discretizing Eqs.~\eqref{eq:ewald2d-intform1} and \eqref{eq::J02}. Consider the integral
\begin{equation}\label{eq::A.1}
	I=\int_{-\infty}^{\infty }\frac{e^{-(a^2+t^2)}}{a^2+t^2}e^{\i b t}dt,
\end{equation}
where $a\geq 0$ and $b\in\mathbb{R}$. We approximate $I$ using a $(2M+1)$-point trapezoidal rule:
\begin{equation}
	I_{M,\xi}=\xi\sum_{j=-M}^{M}\frac{e^{-a^2-(j\xi)^2}}{a^2+(j\xi)^2}e^{\i b j\xi},
\end{equation}
where $\xi>0$ is the step size, and we define the remainder as $E_{M,\xi}:=I-I_{M,\xi}$. Note that the integrand of $I$ has two simple poles at $t_{\pm}=\pm \i a$. Let $\Gamma_{\pm}$ be two positively/negatively oriented rectangular contours with vertices $(M+1/2)\xi\pm\i a^*$, $-(M+1/2)\xi\pm\i a^*$, $(M+1/2)\xi$, and $-(M+1/2)\xi$ (see Fig.~\ref{fig:Trapezoidal}). We enforce $a^*>a$ so that $\Gamma_{\pm}$ encloses both the interval $[-M \xi,M \xi]$ and the pole $t_{\pm}$. By following the approach given in~\cite{Donaldson1972SINUA} and applying Cauchy's theorem, we can derive an estimate for $E_{M,\xi}$ in the limit $M\rightarrow \infty$:
\begin{equation}
	E_{M,\xi}=\int_{\Gamma_++\Gamma_-}\frac{e^{-(a^2+t^2)+\i b t}}{a^2+t^2}\varphi(t)dt-2\pi\i \text{Res}\left[\frac{e^{-(a^2+t^2)+\i b t}}{a^2+t^2}\varphi(t),\pm \i a\right],
\end{equation}
where
\begin{align}
	\varphi(t)=\begin{cases}
		\dfrac{1}{1-e^{2\pi\i t/\xi}},\,\quad &\text{Im}(t)<0,\\[1em]
		-\dfrac{1}{1-e^{-2\pi\i t/\xi}},\,\quad&\text{Im}(t)>0,
	\end{cases}
\end{align}
is related to the characteristic function of the trapezoidal rule, and $\text{Res}[f(t),t_0]$ denotes the residue of a function $f$ at a pole $t_0$. Since the contributions from the vertical sides of $\Gamma_{\pm}$ vanish in the limit $M\rightarrow \infty$ and by using residue calculus, we have
\begin{equation}\label{eq::A.5}
	E_{M,\xi}=\left(\int_{-\infty+\i a^*}^{\infty+\i a^*}-\int_{-\infty-\i a^*}^{\infty-\i a^*}\right)\frac{e^{-(a^2+t^2)+\i b t}}{a^2+t^2}\varphi(t)dt+\frac{\pi}{a}\frac{e^{-ab}+e^{ab}}{1-e^{2\pi a/\xi}}.
\end{equation}
In Eq.~\eqref{eq::A.5}, the last term can be considered as the residue correction of the rule, while the remainder integral along $\text{Im}(t)=a^*$ can be estimated as
\begin{equation}
	\begin{split}
		\left|\int_{-\infty+\i a^*}^{\infty+\i a^*}\frac{e^{-(a^2+t^2)+\i b t}}{a^2+t^2}\varphi(t)dt\right|&\leq e^{(a^*)^2-a^2-a^*b-2\pi a^*/\xi}\int_{-\infty}^{\infty}\frac{|a^2+(t+\i a^*)^2|^{-1}e^{-t^2}}{\left|1-e^{2\pi\i t/\xi-2\pi a^*/\xi}\right|}dt\\
		&\leq\frac{\sqrt{\pi}e^{(a^*)^2-a^2-a^*b-2\pi a^*/\xi}}{(a^*)^2-a^2}.
	\end{split}
\end{equation}

To obtain a closed formula, it is necessary to determine the extremum of the exponent under the condition $a^*>a>0$. Since the range of $a$ is not specified, we can safely choose $a^*=\pi/\xi+b/2$ if $\pi/\xi+b/2>a$, and $a^*=\sqrt{a^2+1}$ otherwise. This choice ensures a decay rate of at least $\sim\mathcal{O}(e^{-\text{sign}(\pi/\xi+b/2)|\pi/\xi+b/2|^2})$, where $\text{sign}(t)=1$ if $t> 0$, $\text{sign}(t)=0$ if $t=0$, and $\text{sign}(t)=-1$ otherwise. By following a similar procedure, we can derive that the integral along $\text{Im}(t)=-a^*$ in Eq.~\eqref{eq::A.5} decays with order $\mathcal{O}(e^{-\text{sign}(\pi/\xi-b/2)|\pi/\xi-b/2|^2})$. Consequently, we can conclude that
\begin{equation}
	E_{M,\xi}=\frac{\pi}{a}\frac{e^{-ab}+e^{ab}}{1-e^{2\pi a/\xi}}+E_{\text{err}},
\end{equation}
where the remainder error term can be estimated as 
\begin{equation}\label{eq::A.8}
	|E_{\text{err}}|\sim \mathcal{O}(e^{-\text{sign}(\pi/\xi-|b|/2)\left|\pi/\xi-|b|/2\right|^2}).
\end{equation}


\bibliographystyle{siamplain}
\bibliography{new}
\end{document}